\documentclass[11pt]{article}%
\usepackage{amssymb,amsmath,amsfonts,amsthm,array,bm,bbm,color}%
\usepackage{amsmath}%
\usepackage{amsfonts}%
\usepackage{amssymb}%
\usepackage{graphicx}

\usepackage{algorithm2e}
\usepackage{subcaption}
\usepackage{dsfont}

\providecommand{\U}[1]{\protect\rule{.1in}{.1in}}

\numberwithin{equation}{section}

\newtheorem{theorem}{Theorem}[section]
\newtheorem{lemma}[theorem]{Lemma}
\newtheorem{corollary}[theorem]{Corollary}
\newtheorem{proposition}[theorem]{Proposition}
\newtheorem{remark}[theorem]{Remark}
\newtheorem{example}[theorem]{Example}

\newtheorem{hypothesis}[theorem]{Hypothesis}

\def\<{\langle}
\def\>{\rangle}
\def\d{{\rm d}}

\def\E{\mathbb{E}}

\def\N{\mathbb{N}}
\def\P{\mathbb{P}}
\def\R{\mathbb{R}}

\newcommand\abs[1]{\left| #1 \right|}
\newcommand\norm[1]{\left\| #1 \right\|}

\begin{document}

\title{A numerical approach to Kolmogorov equation in high dimension based on Gaussian analysis}

\author{Franco Flandoli\footnote{Email: franco.flandoli@sns.it. Scuola Normale Superiore, Piazza dei Cavalieri, 7, 56126 Pisa, Italy.}
\quad Dejun Luo\footnote{Email: luodj@amss.ac.cn. RCSDS, Academy of Mathematics and Systems Science, Chinese Academy of Sciences, Beijing 100190, China, and School of Mathematical Sciences, University of the Chinese Academy of Sciences, Beijing 100049, China.}
\quad Cristiano Ricci\footnote{Email: cristiano.ricci@unifi.it. University of Florence, Italy.}}
\maketitle

\begin{abstract}
For Kolmogorov equations associated to finite dimensional stochastic
differential equations (SDEs) in high dimension, a numerical method
alternative to Monte Carlo simulations is proposed. The structure of the SDE
is inspired by stochastic Partial Differential Equations (SPDE) and thus
contains an underlying Gaussian process which is the key of the algorithm. A
series development of the solution in terms of iterated integrals of the
Gaussian process is given, it is proved to converge - also in the infinite
dimensional limit - and it is numerically tested in a number of examples.
\end{abstract}

\textbf{Keywords:} Kolmogorov equation, numerical solution, iteration scheme, Gaussian process

\section{Introduction}

Kolmogorov equations are parabolic equations with a structure directly related
to stochastic differential equations (SDEs). The SDEs considered here are in a
finite dimensional space but they are inspired by the spatial discretization
of stochastic Partial Differential Equations (SPDE). When the noise is
additive and the nonlinearity is time-independent, a general form of such SDEs
is%
\begin{equation}\label{eq:X^x_t}
\left\{
\aligned
\d X_{t} &  =\left(  AX_{t}+B\left(  X_{t}\right)  \right)  \d t+\sigma\sqrt {Q}\,\d W_{t},\\
X_0 &  =x,
\endaligned \right.
\end{equation}
where $x\in\mathbb{R}^{d}$, $\left(  W_{t}\right)  _{t\geq0}$ is a Brownian
motion in $\mathbb{R}^{d}$ (namely $W_{t}=\left(  W_{t}^{1},\ldots, W_{t}%
^{d}\right)  $ where the $W_{t}^{i}$'s are independent real valued Brownian
motions), defined on a probability space $\left(  \Omega,\mathcal{F}%
,\mathbb{P}\right)  $ with a filtration $\left(  \mathcal{F}_{t}\right)
_{t\geq0}$, $\sigma$ is a positive real number measuring the strength of the
noise, $Q$ is a $d\times d$ positive definite symmetric matrix (the so called
covariance matrix of the noise) describing the spatial structure of the noise
and $\sqrt{Q}$ is its square root, $A$ is a $d\times d$ matrix and
$B:\mathbb{R}^{d}\rightarrow\mathbb{R}^{d}$ is a function with the degree of
regularity specified below. Obviously we could include the scalar $\sigma^{2}$
inside the matrix $Q$ but for certain practical arguments it is useful to
distinguish between them. The solution $X_{t}$ is a continuous
adapted process in $\mathbb{R}^{d}$. The associated Kolmogorov equation is%
\begin{equation}\label{KolEq}
\left\{  \aligned
\partial_{t}u (  t,x )   &  =\frac{\sigma^{2}}{2} {\rm Tr}\left(
QD^{2}u ( t,x)  \right)  +\left\langle Ax+B\left(  x\right)
,Du (  t,x)  \right\rangle, \\
u( 0,x)   &  =u_{0} ( x),
\endaligned \right.
\end{equation}
where $u:[  0,T]  \times\mathbb{R}^{d}\rightarrow\mathbb{R}$,
$Du(  t,x)  $ and $D^{2}u(  t,x)  $ denote respectively
the vector of first partial derivatives and the matrix of second partial
derivatives, ${\rm Tr}\left(  QD^{2}u(  t,x)  \right)  $ is the trace of
the $d\times d$ matrix $QD^{2}u (  t,x )  $ and $\langle
\cdot,\cdot\rangle $ denotes the scalar product in $\mathbb{R}^{d}$.
Both for the SDE and the Kolmogorov equation we have used notations which may
be adapted to the infinite dimensional case, when $\mathbb{R}^{d}$ is replaced
by a Hilbert space (see Section \ref{sec-theory} for the general theory); however, the aim of this work is numerical and all objects
in the introduction will belong to $\mathbb{R}^{d}$. The link between the Kolmogorov equation and
the SDE is%
\[
u (  t,x )  =\mathbb{E}\left[  u_{0} (  X_{t}^{x} )
\right],
\]
where $\mathbb{E}$\ denotes the mathematical expectation on $\left(
\Omega,\mathcal{F},\mathbb{P}\right)  $ and $X_{t}^{x}$ is the solution of the
SDE above, where the initial condition $x$ is explicitly indicated.
Several elements of theory both in finite and infinite dimensions for SDEs and
associated Kolmogorov equations can be found in many books, like
\cite{Cerrai, DaPrato, DaPratoZab2, KarSh, KRZ}.

Solving the Kolmogorov equation with suitable initial condition $u_{0}$ is a
way to compute relevant expected values and probabilities associated to the
solution of\ an SDE. For instance, when $u_{0} (  x )  =1_{\left\{
\left\Vert x\right\Vert >R\right\}  }$, $u (  t,x )  $ is the
probability that the solution exceeds a threshold $R$:
\[
u (  t,x )  =\mathbb{E}\left[  1_{\left\{  \left\Vert x\right\Vert
>R\right\}  }\left(  X_{t}^{x}\right)  \right]  =\mathbb{P}\left(  \left\Vert
X_{t}^{x}\right\Vert >R\right)  .
\]
The classical method of computing these expected values is the Monte Carlo
method (with important variants, see for instance \cite{GNR, Szpruch}):\ several
realizations of the process $X_{t}^{x}$ are simulated by solving the SDE --
typically by Euler method -- and then the corresponding values of $u_{0} (
X_{t}^{x} )  $ are averaged. Going beyond this strategy is a fundamental
issue, due to its limitations in relevant applications like Geophysics and
Climate change projections \cite{Kalnay}, especially concerning extreme events.
The question is whether Kolmogorov equation can be efficiently solved
numerically without using the simulation of the SDE. But the problem is that
the dimension $d$ is extremely high in these examples and common numerical
methods for solution of parabolic equations already require strong
computational power when $d=3$, \cite{Brenner, Kushner}. A grid of $N$ points in $\mathbb{R}$, repeated
for all dimensions, give rise to $N^{d}$ grid points, numerically impossible
when, for instance, $N=10,\, d=10$ (which still would be an extremely poor
approximation). Spectral methods seem to meet the same restrictions:\ $N^{d}$
is the cardinality of basis elements obtained by tensorization of $N$ basis
elements for each space variable.

The problem of dimensionality, the limitations of present methodologies and
several motivations are recalled in two recent works \cite{Jentzen, Jentzen2}
which also aim to go beyond Monte Carlo and propose a method
based on deep artificial neural networks. We address to these brilliant works
for other comments on the problem, see also \cite[Introduction]{ChenMajda}. The approach developed here is however
completely different.

Our aim is to take advantage of the probabilistic structure of the problem to
devise numerical schemes for the Kolmogorov equation, in particular using
Gaussian analysis. We implement a perturbative scheme which links the solution
of Kolmogorov equation to a Gaussian process, the solution $Z_{t}$ of the
linear stochastic equation%
\begin{equation*}
\left\{ \aligned
\d Z_{t} &  =AZ_{t}\,\d t+\sqrt{Q}\,\d W_{t}, \\
Z_0 &  =0. 
\endaligned \right.
\end{equation*}
The idea comes from the theoretical investigations of infinite dimensional
Kolmogorov equations associated to SPDEs, see for instance \cite{DaPratoZab2, DaPratoFla}.
We modify and adapt that idea giving an explicit formula in
terms of a series of Gaussian integrals. We provide here a first glance at the
strategy by writing the final formula:%
\[
u( t,x)  =\sum_{n=0}^{\infty}v^{n} (  t,x),
\]
where%
\[
v^{0} (  t,x)  =\mathbb{E}\left[  u_{0}\left(  e^{tA}x+\sigma
Z_{t} \right)  \right]
\]
and for $n\geq1$%
$$ \aligned
v^{n} (  t,x )  =&\, \int_{0}^{t} \d r_{n}\int_{0}^{r_{n}} \d r_{n-1}\cdots \int_{0}^{r_{2}} \d r_{1} \\
&\ \mathbb{E}\left[  u_{0}\big(  e^{tA}x+\sigma Z_{t} \big)
\prod_{i=1}^{n} \left\langle \Xi_{\sigma} ( r_{i+1}-r_{i} )  B\big(  e^{r_i A}x+\sigma
Z_{r_i} \big)  ,Z_{r_{i+1}}-e^{(  r_{i+1}-r_{i})  A}Z_{r_{i}%
}\right\rangle \right]  .
\endaligned $$
The matrix $\Xi_{\sigma} (t)$ will be defined in the next
sections, see \eqref{sec-2.1}; it is easily computed by $A$ and $Q$, and it depends on the
parameters $t$ and $\sigma$. A theoretical analysis of this series is made,
proving the following result.

\begin{theorem}
Assume that $u_0$ and $B$ are bounded. Then, under suitable conditions on $A$ and $Q$ (see Hypothesis \ref{hypothe} for details), we have the following uniform estimate:
  $$\|v^{n} (t)\|_\infty \leq \|u_0\|_\infty \|B\|_\infty^n C_\delta^n\, t^{n(1-\delta)} \frac{\Gamma(1-\delta)^n} {\Gamma(1+n(1-\delta))}, \quad t>0, $$
where $\Gamma(\cdot)$ is the Gamma function, $C_\delta>0$ is a constant and $\delta\in (0,1)$ the parameter in (iv) of Hypothesis \ref{hypothe}.
\end{theorem}

This theorem sustains the numerical method and stresses the independence on
the dimension of certain issues of the method (obviously others, like getting
a sample of $Z$, have a cost which increases with $d$). When $\mathbb{R}^{d}$
is replaced by a Hilbert space $H$ (and below we shall formulate the theorem
with assumptions in a Hilbert space) it contains also some theoretical
novelties with respect to the literature, especially because it provides an
explicit formula.

The numerical evaluation of the terms $v^{n}( t,x)  $ is made
here, in this paper, by Monte Carlo method based on a sample of the process
$Z_{t}$ obtained by solving the linear SDE by Euler method. These are the most
obvious choices, but other possibilities exist, since $\left(  Z_{t}\right)
_{t\geq0}$ is a centered Gaussian process with known covariance function. A
main strategy invoked here is to store once for ever a large and accurate
sample of $\left(  Z_{t}\right)  _{t\geq0}$ (this requires the pair $\left(
A,Q\right)  $ to be given) and use it later in the formula for different
values of the other parameters, $t,x,\sigma,u_{0}$ and even $B$.

This new method is aimed to replace direct Monte Carlo simulations. We should
therefore accurately compare them. If the purpose is to make one single
computation, classical Monte Carlo wins: the Gaussian method above still
requires Monte Carlo simulations of the linear problem, which is less
expensive than the nonlinear one but then one has to compute possibly several
terms $v^{n} (  t,x)  $;\ some experiments clearly show that
classical Monte Carlo is less expensive for a comparable degree of precision.
The advantage comes when we want to vary parameters, since the Gaussian method
for given $\left(  A,Q\right)  $ allows to store a possibly expensive sample
of the process $Z_{t}$ and reuse it for several values of the parameters,
just having to compute the averages over the Gaussian sample which give us the
terms $v^{n} (  t,x )  $. On the contrary, classical Monte Carlo
method requires to repeat the simulation of the nonlinear problem for each new
value of the parameters. By ``parameters'', as we have already mentioned above,
we mean $t,x,\sigma,u_{0},B$. Let us comment on the interest in changing them.

The interest in changing $t$ is obvious. In certain applications it is
necessary to change the initial condition $x$ and compare or collect the
results. We have in mind for instance the ensemble methods used in weather
prediction where the initial condition is uncertain, a first guess is made on
the basis of physical observations, but then the initial condition is
perturbed in various directions and the final results averaged by suitable
methods. See also \cite{Jentzen, Jentzen2}, where the need to change
$\left(  t,x\right)  $ is stressed.

Changing the strength $\sigma$ of the noise is a very important issue, related
also to Large Deviation Theory. We have to advise that the precision of our
simulations degenerates as $\sigma\rightarrow0$, or the number of iterates
needed to maintain a reasonable precision blows-up, but at least one can
detect some tendency by moving $\sigma$ in a finite range without arriving to
too small values.

Concerning the change of function $u_{0}$, unfortunately the main comment is
in favor of Monte Carlo:\ having at disposal a sample of the process
$X_{t}^{x}$ immediately gives a way to compute $\mathbb{E}\left[  u_{0} (
X_{t}^{x} )  \right]  $ for different functions $u_{0}$. Hence the best
we can say on this issue is that our formula allows for such computations with
a moderate additional effort -- but not with an improvement over Monte Carlo.

Finally, changing the nonlinearity $B$ is of theoretical interest for the
investigation of the performances of the method, and in applications it may be
of interest in those -- very common -- cases when some parameters of $B$ are not
precisely known and different simulations may be useful for comparison or for
ensemble averaging methods performed over the range of those parameters.

Let us finally come to a brief description of numerical
results. In Section 3, we present some numerical results based on the method proposed here in the finite dimensional settings with $d \geq 10$. The results, even if not fully satisfactory yet, should be compared with the fact that the innovative attempts to solve the Kolmogorov equation in $d>3$ by direct methods, see \cite{ChenMajda}, are often restricted to dimensions smaller than $10$. Large dimension is therefore a very difficult problem that deserves strong effort for improvement, and some of our results -- although not in all examples -- are quite promising.

As a final comment, let us explicitly mention that the class of Kolmogorov
equations studied here is particular, because of the additive and very
non-degenerate noise and because we have treated only relatively mild
nonlinearities. We have not considered relevant cases from fluid mechanics
which have more severe nonlinearities and activation of more scales; after a
few initial tests on dyadic models -- we point in particular to the recent
models on trees which may be very relevant for turbulence theory, see
\cite{Barbato, Bianchi, Bianchi Morandin} -- it was clear that
covering these examples with this approach requires further research and
improvements. Extension to multiplicative transport noises \cite{FlaLuo, FlaLuo2} is another challenging open question.

\section{The iteration scheme for Kolmogorov equations on Hilbert spaces}\label{sec-theory}

In this section we work in an infinite dimensional separable Hilbert space $H$ and study the iteration scheme for the Kolmogorov equation:
  \begin{equation}\label{KolE}
  \partial_t u(t,x) =\frac12 {\rm Tr}\big( Q D^2 u(t,x) \big) +\big\< Ax+ B(x), D u(t,x) \big\>, \quad u(0,\cdot)= u_0.
  \end{equation}
Here $A:D(A)\subset H\to H$ is an unbounded linear operator, $Q$ is a nonnegative self-adjoint bounded linear operator on $H$, $B: D(B)\subset H\to H$ is a nonlinear measurable mapping and $u_0:H\to \R$ is a real valued measurable function. In this section $Q$ plays the role of $\sigma^2 Q$ to simplify notation. In the following we write $\mathcal L(H,H)$ for the Banach space of bounded linear operators on $H$ with the norm $\|\cdot\|_{\mathcal L(H)}$.

Throughout this section we assume the following conditions:

\begin{hypothesis}\label{hypothe}
\begin{itemize}
\item[\rm (i)] $A:D(A)\subset H\to H$ is the infinitesimal generator of a strongly continuous semigroup $e^{tA}$.
\item[\rm (ii)] $Q$ is a nonnegative self-adjoint operator in $\mathcal L(H,H)$ satisfying ${\rm Ker}(Q)= \{0\}$, and for any $t>0$ the linear operator
  \begin{equation}\label{covar-matrix}
  Q_t= \int_0^t e^{sA} Q e^{sA^\ast} \,\d s
  \end{equation}
is of trace class.
\item[\rm (iii)] We have $e^{tA}(H) \subset Q_t^{1/2}(H)$ for any $t>0$.
\item[\rm (iv)] Letting $\Lambda(t) = Q_t^{-1/2} e^{tA}$, we assume there exist $\delta\in (0,1)$ and $C_\delta>0$ such that
  $$\|\Lambda(t) \|_{\mathcal L(H)} \leq C_\delta /t^\delta, \quad t>0 .$$
\end{itemize}
\end{hypothesis}

The assumptions (i)--(iii) are quite standard in the literature, see for instance \cite[Hypothesis 2.1 and 2.24]{DaPrato}. The operator $\Xi_\sigma (t)$ appeared in the introduction has the form
  \begin{equation}\label{sec-2.1}
  \Xi_\sigma (t) = \sigma Q_t^{-1/2} \Lambda(t) = \sigma Q_t^{-1} e^{tA};
  \end{equation}
we remark that, in the setting of the introduction, the operator $Q$ in \eqref{covar-matrix} should be replaced by $\sigma^2 Q$ when computing $Q_t$. The following example is taken from \cite[Example 2.5]{DaPrato} which verifies all the assumptions.

\begin{example}\label{ex:laplacian}
Let $\mathcal O=[0,\pi]^d$ with $d\in \N$. We choose $H= L^2(\mathcal O)$, and
  $$Ax= \Delta x, \quad x\in D(A)= H^2(\mathcal O) \cap H^1_0(\mathcal O),$$
where $\Delta$ is the Laplacian operator with Dirichlet boundary condition. $A$ is a self-adjoint negative operator in $H$, and
  $$A e_k= -|k|^2 e_k, \quad  k\in \N^d,$$
where for $k\in \N^d$, $|k|^2 = k_1^2 + \cdots + k_d^2$ and
  $$e_k(\xi)= (2/\pi)^{d/2} \sin(k_1\xi_1) \cdots \sin(k_d\xi_d), \quad \xi \in [0,\pi]^d. $$

Choose $Q= (-A)^{-\alpha},\, \alpha\in [0,1)$, so that
  $$Q x= \sum_{k\in \N^d} |k|^{-2\alpha} \<x, e_k\> e_k, \quad x\in H.$$
For any $t>0$, if $\alpha > d/2-1$, then
  $${\rm Tr}(Q_t) =\sum_{k\in \N^d} \frac1{2 |k|^{2+2\alpha}} \Big(1- e^{-2t|k|^2} \Big) <\infty.$$
So (ii) is satisfied.

Next, (iii) can be checked by explicit computations. Moreover,
  $$\Lambda(t) x= \sum_{k\in \N^d} \frac{\sqrt{2}\, |k|^{1+\alpha}}{\sqrt{e^{2t|k|^2} -1}} \<x, e_k\> e_k, \quad x\in H.$$
From this we deduce that
  $$\|\Lambda(t) \|_{\mathcal L(H)} \leq \frac{\sqrt{2C_\alpha}}{t^{(1+\alpha)/2}}, $$
where
  $$C_\alpha= \sup_{\theta>0} \frac{\theta^{1+\alpha}}{e^{2\theta}- 1} <+\infty. $$
Thus (iv) holds with $\delta = (1+\alpha)/2 \in [1/2,1)$.
\end{example}

We also need the following technical conditions.

\begin{hypothesis}\label{hypothe-1}
The initial datum $u_0:H\to \R$ and the nonlinear part $B:H\to H$ in \eqref{KolE} are bounded and measurable.
\end{hypothesis}

This section is organized as follows. In Subsection \ref{subsec-derivative-formula}, we recall some basic facts in Gaussian analysis on Hilbert space and give the formula for the first term $v^1(t,x)$ of the iteration \eqref{new-iteration}. We give in Section \ref{subsec-second-term} the details for calculating the second term $v^2(t,x)$, which will help us to guess and prove the formula for general terms $v^n(t,x)$ in Section \ref{subsec-general-term}. In the last part, we estimate the uniform norm of $v^n(t,x)$ and show the convergence of the iteration scheme. The limit is the unique mild solution of \eqref{KolE}, see Theorem \ref{thm-main-result}.

\subsection{Some preparations}\label{subsec-derivative-formula}

Let $W$ be a cylindrical Brownian motion on $H$:
  $$W_t = \sum_{k=1}^\infty W^k_t e_k,\quad t\geq 0,$$
where $\{e_k\}_{k\geq 1}$ is a complete orthonormal basis of $H$ and $\{W^k \}_{k\geq 1}$ is a family of independent one dimensional standard Brownian motions defined on some probability space $(\Omega, \mathcal F, \P)$. Under the conditions (i) and (ii) in Hypothesis \ref{hypothe}, the linear SDE
  \begin{equation}\label{linear-SDE}
  \d Z_t^x= AZ_t^x\,\d t + \sqrt{Q}\,\d W_t, \quad Z_0^x=x\in H
  \end{equation}
has a unique solution with the expression
  $$Z_t^x = e^{tA} x + W_A(t),\quad t>0,$$
where $W_A(t)$ is the stochastic convolution:
  $$W_A(t)=\int_0^t e^{(t-s)A} \sqrt{Q}\,\d W_s. $$
For any $t>0$, $W_A(t)$ is a centered Gaussian variable on $H$ with covariance operator $Q_t$. We denote its law by $N_{Q_t}(\d y)$. Accordingly, the law of $Z_t^x$ is denoted as $N_{e^{tA} x, Q_t}(\d y)$. Recall that for any $h\in H$, $\big\<h, Q_t^{-1/2}W_A(t) \big>$ is a centered real Gaussian variable with variance
  \begin{equation*}
  \E \big\<h, Q_t^{-1/2}W_A(t) \big>^2 =|h|_H^2.
  \end{equation*}

We shall write $\mathcal B(H)$ for the space of bounded measurable functions on $H$ and $C_b^1(H)$ the space of Fr\'echet differentiable functions, bounded with bounded derivatives. When $f\in C_b^1(H)$, its Fr\'echet derivative will be denoted by $D f$. For any $f\in \mathcal B(H)$ and $t\geq 0$, let
  \begin{equation*}
  S_t f(x):= \E f(Z_t^x) = \int_{H} f(y)\, N_{e^{tA} x, Q_t}(\d y) = \int_H f\big( e^{tA} x +y \big)\, N_{Q_t}(\d y).
  \end{equation*}
This defines a Markov semigroup on $H$. We have the following important result which implies $S_t$ is strong Feller (see \cite[Proposition 2.28]{DaPrato} for a proof).

\begin{proposition}\label{derivative-formula}
Assume the conditions (i)--(iii) in Hypothesis \ref{hypothe}. Then for all $f\in \mathcal B(H)$ and $t>0$, we have $S_t f \in C_b^1(H)$ and for any $h\in H$,
  \begin{equation}\label{derivative-formula.1}
  \<h, D S_tf(x)\> = \E \big[ f(Z_t^x) \big\<\Lambda(t) h, Q_t^{-1/2} \big(Z_t^x- e^{tA}x \big) \big\>\big].
  \end{equation}
Moreover,
  \begin{equation}\label{derivative-formula.2}
  \| D S_tf \|_\infty \leq \|f\|_\infty \|\Lambda(t)\|_{\mathcal L(H)}.
  \end{equation}
\end{proposition}

Using the semigroup $S_t$, the \emph{mild} formulation of the Kolmogorov equation \eqref{KolE} is
  \begin{equation}\label{mild-sol}
  u(t,x)= (S_t u_0)(x) + \int_0^t \big(S_{t-s} \< B, D u(s) \>\big)(x)\,\d s.
  \end{equation}
This suggests us to consider the iterative scheme:
  $$u^{n+1}(t,x) = (S_t u_0)(x)+ \int_0^t \big(S_{t-s} \< B, D u^n(s) \>\big)(x)\,\d s$$
with $u^0(t,x) = (S_t u_0)(x) = \E u_0(Z_t^x)$. We define $v^0(t,x) =u^0(t,x)$ and
  $$v^n(t,x) = u^n(t,x)- u^{n-1}(t,x),\quad n\geq 1,$$
then the new functions satisfy the iteration procedure:
  \begin{equation}\label{new-iteration}
  \left\{ \aligned
  v^{n+1}(t,x) &= \int_0^t (S_{t-s} k^n_s)(x)\,\d s, \\
  k^n_s(y) &= \< B(y), D v^n(s,y) \>, \\
  v^0(t,x)&= \E u_0(Z_t^x).
  \endaligned
  \right.
  \end{equation}

Before concluding this section, we show how to obtain the first term $v^1(t,x)$. Since $u_0\in \mathcal B(H)$, Proposition \ref{derivative-formula} implies $v^0(t)\in C_b^1(H)$ for any $t>0$, and thus $\< B, D v^0(t) \> \in \mathcal B(H)$. Denote by $\mathcal F_t$ the filtration generated by the cylindrical Brownian motion $W_t$.

\begin{lemma}\label{lem-first-iteration}
It holds that
  $$\big(S_{t-s} k^0_s \big)(x) = \E\Big[u_0(Z_t^x) \big\<\Lambda(s) B(Z_{t-s}^x), Q_s^{-1/2}\big(Z_t^x - e^{sA} Z_{t-s}^x \big) \big\> \Big].$$
\end{lemma}

\begin{proof}
Use the property of conditional expectation:
  $$\aligned
  &\ \E \Big[u_0(Z_t^x) \big\<\Lambda(s) B(Z_{t-s}^x), Q_s^{-1/2} (Z_t^x- e^{sA} Z_{t-s}^x) \big\> \Big]\\
  =&\ \E\Big\{\E \Big[u_0(Z_t^x) \big\<\Lambda(s) B(Z_{t-s}^x), Q_s^{-1/2} (Z_t^x- e^{sA} Z_{t-s}^x) \big\> \big| \mathcal F_{t-s} \Big] \Big\} \\
  =&\ \E\Big\{\E \Big[u_0(Z_t^x) \big\<\Lambda(s) B(Z_{t-s}^x), Q_s^{-1/2} (Z_t^x- e^{sA} Z_{t-s}^x) \big\> \big| Z_{t-s}^x \Big] \Big\},
  \endaligned $$
where the second step follows from the Markov property. Again by the Markov property,
  $$\aligned
  &\ \E \Big[u_0(Z_t^x) \big\<\Lambda(s) B(Z_{t-s}^x), Q_s^{-1/2} (Z_t^x- e^{sA} Z_{t-s}^x) \big\> \big| Z_{t-s}^x \Big]\\
  =&\ \E \Big[u_0(Z_s^y) \big\<\Lambda(s) B(y), Q_s^{-1/2} (Z_s^y- e^{sA} y) \big\> \Big]_{y= Z_{t-s}^x }\\
  =&\ k^0_s(y) \big|_{y= Z_{t-s}^x }= k^0_s(Z_{t-s}^x),
  \endaligned $$
where the second step is due to \eqref{derivative-formula.1}. Substituting this equality into the previous one we obtain the identity.
\end{proof}

The above lemma implies

\begin{corollary}\label{cor-first-term}
For any $t>0$ and $x\in H$,
  \begin{equation}\label{cor-first-term.1}
  v^1(t,x) = \int_0^t \E\Big[u_0(Z_t^x) \big\<\Lambda(s) B(Z_{t-s}^x), Q_s^{-1/2}\big(Z_t^x - e^{sA} Z_{t-s}^x \big) \big\> \Big]\,\d s.
  \end{equation}
Moreover,
  $$\|v^1(t)\|_\infty \leq \|u_0\|_\infty \|B\|_\infty \int_0^t \|\Lambda(s)\|_{\mathcal L(H)}\,\d s$$
and
  $$\|D v^1(t)\|_\infty \leq \|u_0\|_\infty \|B\|_\infty \int_0^t \|\Lambda(t-s)\|_{\mathcal L(H)} \|\Lambda(s)\|_{\mathcal L(H)}\,\d s. $$
\end{corollary}

\begin{proof}
The formula \eqref{cor-first-term.1} follows directly from Lemma \ref{lem-first-iteration}. Next, by the definition \eqref{new-iteration} of the iteration, for any $s>0$ and $y\in H$,
  \begin{equation}\label{cor-first-term.2}
  \big|k^0_s(y)\big| \leq |B(y)| \, |D v^0(s,y)| \leq \|B\|_\infty |D S_s u_0(y)| \leq \|B\|_\infty \|u_0\|_\infty \|\Lambda(s)\|_{\mathcal L(H)},
  \end{equation}
where the last inequality follows from \eqref{derivative-formula.2}. Therefore,
  $$|v^1(t,x)| \leq \int_0^t \big| \big(S_{t-s} k^0_s\big)(x)\big| \,\d s \leq \int_0^t \big\| k^0_s \big\|_\infty \,\d s \leq \|u_0\|_\infty \|B\|_\infty \int_0^t \|\Lambda(s)\|_{\mathcal L(H)}\,\d s $$
which yields the estimate on $\|v^1(t)\|_\infty$. The inequality \eqref{cor-first-term.2} implies that $k^0_s\in \mathcal B(H)$ for all $s>0$, hence by Proposition \ref{derivative-formula}, $S_{t-s} k^0_s \in C_b^1(H)$ and
  $$D v^1(t,x) = \int_0^t D \big(S_{t-s} k^0_s\big)(x) \,\d s.$$
Finally, by \eqref{derivative-formula.2},
  $$\|D v^1(t)\|_\infty \leq \int_0^t \big\|D \big(S_{t-s} k^0_s\big)\big\|_\infty \,\d s \leq \int_0^t \big\|k^0_s \big\|_\infty \|\Lambda(t-s) \|_{\mathcal L(H)}\,\d s,$$
which, together with \eqref{cor-first-term.2}, gives us the last estimate.
\end{proof}

\subsection{The term $v^2(t,x)$} \label{subsec-second-term}

In this part, we compute the second term in the iteration to illustrate the ideas. First we prove

\begin{lemma}\label{lem-sec-term-2}
One has
\begin{align*}
k_{t}^{1}( x) & =\int_0^t \E\Big[ u_0( Z_t^x) \big\<\Lambda(s) B(Z_{t-s}^x), Q_s^{-1/2} (Z_{t}^x- e^{sA} Z_{t-s}^x) \big\> \\
  &\hskip40pt \times \big\<\Lambda(t-s) B(x), Q_{t-s}^{-1/2} (Z_{t-s}^x- e^{(t-s)A}x ) \big\> \Big] \,\d s.
\end{align*}
\end{lemma}

\begin{proof}
By Corollary \ref{cor-first-term}, for any $t>0$, $v^1(t)\in C_b^1(H)$ and
\[\aligned
k_{t}^{1}( x) &=\left\langle B(x), D v^{1} (t,x)  \right\rangle =\, \int_{0}^{t}\big\langle B( x),D \big( S_{t-s} k^0_s\big)(x)  \big\rangle\, \d s.
\endaligned \]%
Recall that \eqref{cor-first-term.2} implies $k^0_s\in \mathcal B(H)$, thus by Proposition \ref{derivative-formula},
  $$k_{t}^{1}( x) =\int_{0}^{t} \E\Big[ k^0_s(Z^x_{t-s}) \big\langle \Lambda(t-s) B(x), Q_{t-s}^{-1/2} \big( Z^x_{t-s} - e^{(t-s) A} x \big) \big\rangle \Big]\, \d s. $$
According to the proof of Lemma \ref{lem-first-iteration}, we have
  $$k^0_s(Z^x_{t-s})= \E \Big[u_0(Z_t^x) \big\<\Lambda(s) B(Z_{t-s}^x), Q_s^{-1/2} (Z_t^x- e^{sA} Z_{t-s}^x) \big\> \big| \mathcal F_{t-s} \Big].$$
Note that $\big\langle \Lambda(t-s) B(x), Q_{t-s}^{-1/2} \big( Z^x_{t-s} - e^{(t-s) A} x \big) \big\rangle$ is $\mathcal F_{t-s}$-measurable. Substituting this equality into the one above and using the property of conditional expectation, we obtain the desired result.
\end{proof}

Now we are ready to present the expression and estimates for the second iteration.

\begin{proposition}\label{prop-second-iteration}
For any $t>0$ and $x\in H$,
\begin{align*}
  v^{2}(t,x) & =\int_{0}^{t}\! \int_{0}^{s} \E\Big[ u_0( Z_t^x) \big\<\Lambda(r) B(Z_{t-r}^x), Q_r^{-1/2} (Z_t^x - e^{rA} Z_{t-r}^x ) \big\> \\
  &\hskip60pt \times \big\<\Lambda(s-r) B(Z^x_{t-s}), Q_{s-r}^{-1/2} (Z_{t-r}^x- e^{(s-r)A} Z^x_{t-s} ) \big\> \Big] \,\d r\d s.
\end{align*}
Furthermore,
  $$\|v^2(t)\|_\infty \leq \|u_0\|_\infty \|B\|_\infty^2 \int_{0}^{t}\! \int_{0}^{s} \|\Lambda(s-r)\|_{\mathcal L(H)} \|\Lambda(r)\|_{\mathcal L(H)}\,\d r\d s$$
and
  $$\|D v^2(t)\|_\infty \leq \|u_0\|_\infty \|B\|_\infty^2 \int_{0}^{t}\! \int_{0}^{s} \|\Lambda(t-s)\|_{\mathcal L(H)} \|\Lambda(s-r)\|_{\mathcal L(H)} \|\Lambda(r)\|_{\mathcal L(H)}\,\d r\d s. $$
\end{proposition}

\begin{proof}
By Lemma \ref{lem-sec-term-2}, for any $s>0$ and $y\in H$,
\begin{align*}
k_{s}^{1}( y)& =\int_0^s \E\Big[ u_0( Z_s^y) \big\<\Lambda(r) B(Z_{s-r}^y), Q_r^{-1/2} (Z_{s}^y- e^{rA} Z_{s-r}^y ) \big\> \\
  &\hskip40pt \times \big\<\Lambda(s-r) B(y), Q_{s-r}^{-1/2} (Z_{s-r}^y- e^{(s-r)A} y ) \big\> \Big] \,\d r.
\end{align*}
We have%
  $$\aligned
  \E \big[ k^1_s(Z^x_{t-s}) \big] &= \E \bigg\{ \int_0^s \E\Big[ u_0( Z_s^y) \big\<\Lambda(r) B(Z_{s-r}^y), Q_r^{-1/2} (Z_{s}^y- e^{rA} Z_{s-r}^y ) \big\> \\
  &\hskip60pt \times \big\<\Lambda(s-r) B(y), Q_{s-r}^{-1/2} (Z_{s-r}^y- e^{(s-r)A} y ) \big\> \Big]_{y= Z^x_{t-s}} \,\d r \bigg\} \\
  &= \int_0^s \E\Big[ u_0( Z_t^x) \big\<\Lambda(r) B(Z_{t-r}^x), Q_r^{-1/2} (Z_t^x - e^{rA} Z_{t-r}^x ) \big\> \\
  &\hskip45pt \times \big\<\Lambda(s-r) B(Z^x_{t-s}), Q_{s-r}^{-1/2} (Z_{t-r}^x- e^{(s-r)A} Z^x_{t-s} ) \big\> \Big] \,\d r,
  \endaligned $$
where the second step follows from the Markov property. Therefore,
  \begin{align*}
  v^{2} (t,x ) & =\int_{0}^{t} \big( S_{t-s}k_{s}^{1}\big)( x)\,\d s =\int_{0}^{t}\mathbb{E} \big[ k_{s}^{1}( Z_{t-s}^{x}) \big] \d s \\
  & =\int_{0}^{t}\int_{0}^{s} \E\Big[ u_0( Z_t^x) \big\<\Lambda(r) B(Z_{t-r}^x), Q_r^{-1/2} (Z_t^x - e^{rA} Z_{t-r}^x ) \big\> \\
  &\hskip60pt \times \big\<\Lambda(s-r) B(Z^x_{t-s}), Q_{s-r}^{-1/2} (Z_{t-r}^x- e^{(s-r)A} Z^x_{t-s} ) \big\> \Big] \,\d r\d s.
  \end{align*}

Next, by the definition of $k^1_s$ and the last inequality in Corollary \ref{cor-first-term},
  \begin{equation}\label{prop-second-iteration.1}
  \big\| k^1_s \big\|_\infty \leq \|B\|_\infty \|D v^1(s)\|_\infty \leq \|u_0\|_\infty \|B\|_\infty^2 \int_{0}^{s} \|\Lambda(s-r)\|_{\mathcal L(H)} \|\Lambda(r)\|_{\mathcal L(H)}\,\d r.
  \end{equation}
This immediately implies
  $$\aligned
  |v^{2} (t,x )| &\leq \int_0^t \big\| k^1_s \big\|_\infty \,\d s \leq \|u_0\|_\infty \|B\|_\infty^2 \int_0^t \int_{0}^{s} \|\Lambda(s-r)\|_{\mathcal L(H)} \|\Lambda(r)\|_{\mathcal L(H)}\,\d r \d s,
  \endaligned $$
and we obtain the estimate on $\|v^2(t) \|_\infty$. Moreover, by Proposition \ref{derivative-formula},
  $$|D v^{2} (t,x )| \leq \int_0^t \big| D \big( S_{t-s}k_{s}^{1}\big)( x) \big| \,\d s \leq \int_0^t \big\| k^1_s \big\|_\infty \|\Lambda(t-s) \|_{\mathcal L(H)}\,\d s,$$
which, combined with \eqref{prop-second-iteration.1}, gives us the second estimate.
\end{proof}

\subsection{The general terms $v^n(t,x)$} \label{subsec-general-term}

In order to do further iteration, we rewrite the formula in Proposition \ref{prop-second-iteration} as
  \begin{align*}
  v^{2} (t,x ) & =\int_{0}^{t} \d s_2 \int_{0}^{s_2} \d s_1\, \E\Big[ u_0( Z_t^x) \big\<\Lambda(s_1) B(Z_{t-s_1}^x), Q_{s_1}^{-1/2} (Z_t^x - e^{s_1 A} Z_{t-s_1}^x ) \big\> \\
  &\hskip80pt \times \big\<\Lambda(s_2-s_1) B(Z^x_{t-s_2}), Q_{s_2-s_1}^{-1/2} (Z_{t-s_1}^x- e^{(s_2-s_1)A} Z^x_{t-s_2} ) \big\> \Big] .
  \end{align*}
Moreover, denoting by $s_0=0$, then we have
  $$\aligned
  v^{2} (t,x )=& \int_{0}^{t} \d s_2 \int_{0}^{s_2} \d s_1\\
  &\ \E\Bigg[ u_0( Z_t^x) \prod_{i=1}^2 \Big\<\Lambda(s_i- s_{i-1}) B(Z_{t-s_i}^x), Q_{s_i-s_{i-1}}^{-1/2} \big( Z_{t-s_{i-1}}^x - e^{(s_i- s_{i-1}) A} Z_{t-s_i}^x \big) \Big\> \Bigg].
  \endaligned $$
From this we can guess the general formulae.

\begin{theorem}\label{thm-iteration}
Let $s_0=0$. For any $n\geq 1$,
  \begin{equation}\label{iteration}
  \aligned
  v^{n} (t,x ) =& \int_{0}^{t} \d s_n \int_0^{s_n} \d s_{n-1} \cdots \int_{0}^{s_2} \d s_1\\
  &\ \E\Bigg[ u_0( Z_t^x) \prod_{i=1}^n \Big\<\Lambda(s_i- s_{i-1}) B(Z_{t-s_i}^x), Q_{s_i-s_{i-1}}^{-1/2} \big( Z_{t-s_{i-1}}^x - e^{(s_i- s_{i-1}) A} Z_{t-s_i}^x \big) \Big\> \Bigg].
  \endaligned
  \end{equation}
Moreover,
  $$\|v^n(t)\|_\infty \leq \|u_0\|_\infty \|B\|_\infty^n \int_{0}^{t} \d s_n \int_0^{s_n} \d s_{n-1} \cdots \int_{0}^{s_2} \d s_1 \, \prod_{i=1}^n \|\Lambda(s_i-s_{i-1})\|_{\mathcal L(H)} $$
and, letting $s_{n+1} =t$,
  $$\|D v^n(t)\|_\infty \leq \|u_0\|_\infty \|B\|_\infty^n \int_{0}^{t} \d s_n \int_0^{s_n} \d s_{n-1} \cdots \int_{0}^{s_2} \d s_1 \, \prod_{i=1}^{n+1} \|\Lambda(s_i-s_{i-1})\|_{\mathcal L(H)}. $$
\end{theorem}

\begin{proof}
We proceed by induction. Indeed, in view of the proofs in Section \ref{subsec-second-term}, we shall also prove inductively the formula
  $$\aligned
  k^{n}_t(x)=&\, \int_{0}^{t} \d s_n \int_0^{s_n} \d s_{n-1} \cdots \int_{0}^{s_2} \d s_1 \\
  &\ \E\Bigg[u_0( Z_t^x)\prod_{i=1}^{n+1} \Big\<\Lambda(s_i- s_{i-1}) B(Z_{t-s_i}^x), Q_{s_i-s_{i-1}}^{-1/2} \big( Z_{t-s_{i-1}}^x - e^{(s_i- s_{i-1}) A} Z_{t-s_i}^x \big) \Big\> \Bigg],
  \endaligned $$
where $s_0=0$ and $s_{n+1} =t$. The discussions in Sections \ref{subsec-derivative-formula} and \ref{subsec-second-term} show that the assertions on $v$ hold for $n=1,\, 2$, and the above formula of $k$ holds with $n=1$. Now we assume the assertions on $v$ (resp. on $k$) hold for $n$ (resp. for $n-1$), and try to prove them in the next iteration.

By the induction hypotheses, we have $v^n(s)\in C_b^1(H)$ for all $s>0$ and thus, by the definition of the iteration \eqref{new-iteration}, $k^n_s \in \mathcal B(H)$ with
  $$\aligned
  \big\| k^n_s \big\|_\infty &\leq \|B\|_\infty \|D v^n(s)\|_\infty\\
  &\leq \|u_0\|_\infty \|B\|_\infty^{n+1} \int_{0}^{s} \d s_n \int_0^{s_n} \d s_{n-1} \cdots \int_{0}^{s_2} \d s_1 \, \prod_{i=1}^{n+1} \|\Lambda(s_i-s_{i-1})\|_{\mathcal L(H)},
  \endaligned $$
where $s_{n+1} =s$. Proposition \ref{derivative-formula} implies $S_{t-s} k^n_s\in C_b^1(H)$ for all $s\in (0,t)$, and from the formula
  $$v^{n+1}(t,x) = \int_0^t \big(S_{t-s} k^n_s \big)(x) \,\d s $$
we deduce readily the estimates on $\|v^{n+1}(t)\|_\infty$ and $\|D v^{n+1}(t)\|_\infty$.

Next we prove the formula for $k^n_t(x)$ (note that the induction hypothesis gives us the expression of $k^{n-1}_t(x)$). We have
  \begin{equation}\label{iteration-1}
  \aligned
  k^n_t(x) &= \<B(x), D v^n(t,x)\> = \int_0^t \big\<B(x), D\big( S_{t-s} k^{n-1}_s\big) (x)\big\>\,\d s\\
  &= \int_0^t \E \Big[ k^{n-1}_s(Z^x_{t-s}) \big\< \Lambda(t-s) B(x), Q_{t-s}^{-1/2} (Z^x_{t-s} - e^{(t-s)A} x) \big\> \Big]\,\d s,
  \endaligned
  \end{equation}
where we used Proposition \ref{derivative-formula} in the last step. By the induction hypothesis,
  $$\aligned
  k^{n-1}_s(y)=&\, \int_{0}^{s} \d s_{n-1} \int_0^{s_{n-1}} \d s_{n-2} \cdots \int_{0}^{s_2} \d s_1 \\
  &\ \E\Bigg[u_0( Z_s^y)\prod_{i=1}^{n} \Big\<\Lambda(s_i- s_{i-1}) B(Z_{s-s_i}^y), Q_{s_i-s_{i-1}}^{-1/2} \big( Z_{s-s_{i-1}}^y - e^{(s_i- s_{i-1}) A} Z_{s-s_i}^y \big) \Big\> \Bigg],
  \endaligned $$
where $s_0=0$ and $s_{n} =s$. Therefore, by the Markov property,
  $$\aligned
  &\ k^{n-1}_s(Z^x_{t-s})\\
  =&\, \int_{0}^{s} \d s_{n-1} \int_0^{s_{n-1}} \d s_{n-2} \cdots \int_{0}^{s_2} \d s_1 \\
  &\ \E\Bigg[u_0( Z_t^x)\prod_{i=1}^{n} \Big\<\Lambda(s_i- s_{i-1}) B(Z_{t-s_i}^x), Q_{s_i-s_{i-1}}^{-1/2} \big( Z_{t-s_{i-1}}^x - e^{(s_i- s_{i-1}) A} Z_{t-s_i}^x \big) \Big\> \bigg| \mathcal F_{t-s} \Bigg].
  \endaligned $$
Inserting this identity into \eqref{iteration-1} and noticing that $\big\< \Lambda(t-s) B(x), Q_{t-s}^{-1/2} (Z^x_{t-s} - e^{(t-s)A} x) \big\>$ is measurable with respect to $\mathcal F_{t-s}$, we obtain
  $$\aligned
  k^n_t(x) &= \int_0^t \d s \int_0^s \d s_{n-1} \cdots \int_0^{s_2} \d s_1 \, \E \Bigg\{ \big\< \Lambda(t-s) B(x), Q_{t-s}^{-1/2} (Z^x_{t-s} - e^{(t-s)A} x) \big\> \\
  &\hskip40pt \times u_0( Z_t^x)\prod_{i=1}^{n} \Big\<\Lambda(s_i- s_{i-1}) B(Z_{t-s_i}^x), Q_{s_i-s_{i-1}}^{-1/2} \big( Z_{t-s_{i-1}}^x - e^{(s_i- s_{i-1}) A} Z_{t-s_i}^x \big) \Big\> \Bigg\}.
  \endaligned $$
Renaming $s$ as $s_n$ gives us the formula of $k^n_t(x)$ in the new iteration for all $t>0$ and $x\in H$.

Finally we prove the expression for $v^{n+1}(t,x)$. We have
  $$v^{n+1}(t,x) = \int_0^t \big(S_{t-s} k^n_s \big)(x) \,\d s = \int_0^t \E \big[ k^n_s(Z^x_{t-s}) \big] \,\d s.$$
Using the formula we have just proved for $k^n_s(y)$ and the Markov property, we can obtain the expression for $v^{n+1}(t,x)$ in a similar way as above.
\end{proof}

We give a slightly different formula which is more appropriate for numerical purpose.

\begin{corollary}\label{cor-iteration-new}
For any $n\geq 1$,
  \begin{equation}\label{iteration-new}
  \aligned
  v^{n} (t,x ) =& \int_{0}^{t} \d r_n \int_0^{r_n} \d r_{n-1} \cdots \int_{0}^{r_2} \d r_1\\
  &\ \E\Bigg[ u_0( Z_t^x) \prod_{i=1}^n \Big\<\Lambda(r_{i+1}- r_{i}) B\big(Z_{r_i}^x \big), Q_{r_{i+1}- r_{i}}^{-1/2} \big( Z_{r_{i+1}}^x - e^{(r_{i+1}- r_{i}) A} Z_{r_i}^x \big) \Big\> \Bigg],
  \endaligned
  \end{equation}
where $r_{n+1} =t$. Accordingly,
  $$\aligned
  \|v^n(t)\|_\infty \leq \|u_0\|_\infty \|B\|_\infty^n \int_{0}^{t} \d r_n \int_0^{r_n} \d r_{n-1} \cdots \int_{0}^{r_2} \d r_1 \, \prod_{i=1}^n \|\Lambda(r_{i+1}-r_i)\|_{\mathcal L(H)}
  \endaligned$$
and, setting $r_0=0$,
  $$\|D v^n(t)\|_\infty \leq \|u_0\|_\infty \|B\|_\infty^n \int_{0}^{t} \d r_n \int_0^{r_n} \d r_{n-1} \cdots \int_{0}^{r_2} \d r_1 \, \prod_{i=0}^n \|\Lambda(r_{i+1}-r_i)\|_{\mathcal L(H)}.$$
\end{corollary}

\begin{proof}
We change variables as follows:
  $$r_i = t- s_{n+1-i},\quad 1\leq i\leq n.$$
The domain of integration becomes
  $$\big\{(r_1,\cdots, r_n): 0< r_1 <\cdots< r_n <t \big\}; $$
and $s_i -s_{i-1}= r_{n+2-i} - r_{n+1-i},\, 1\leq i\leq n$. Therefore, by \eqref{iteration},
  $$\aligned
  v^{n} (t,x ) =& \int_{0}^{t} \d r_n \int_0^{r_n} \d r_{n-1} \cdots \int_{0}^{r_2} \d r_1\\
  &\ \E\Bigg[ u_0( Z_t^x) \prod_{i=1}^n \Big\<\Lambda(r_{n+2-i} - r_{n+1-i}) B\big(Z_{r_{n+1-i}}^x \big),\\
  &\hskip50pt Q_{r_{n+2-i} - r_{n+1-i}}^{-1/2} \big( Z_{r_{n+2-i}}^x - e^{(r_{n+2-i} - r_{n+1-i}) A} Z_{r_{n+1-i}}^x \big) \Big\> \Bigg].
  \endaligned $$
In the product, letting $j=n+1- i $, we get the desired formula \eqref{iteration-new}. The proofs of the two estimates are similar.
\end{proof}

\begin{remark}
Due to the convolution structure \eqref{new-iteration}, it seems that \eqref{iteration-new} is not suitable for the induction argument in the proof of Theorem \ref{thm-iteration}.
\end{remark}

\subsection{Convergence of the iteration scheme \eqref{new-iteration}}

We need the following technical result, where we use the Gamma function $\Gamma(\alpha)$:
  $$\Gamma(\alpha) = \int_0^\infty \theta^{\alpha -1} e^{-\theta}\,\d\theta, \quad \alpha >0. $$

\begin{lemma}\label{lem-technical}
Assume $\delta\in (0,1)$ and $n\geq 1$. Let $r_0=0$ and $r_{n+1}=t$. One has
  $$\int_{0}^{t} \d r_n \int_0^{r_n} \d r_{n-1} \cdots \int_{0}^{r_2} \d r_1 \prod_{i=1}^n \frac1{(r_{i+1} -r_{i})^\delta } = \frac{\Gamma(1-\delta)^n} {\Gamma(1+n(1-\delta))} t^{n(1-\delta)} $$
and
  $$\int_{0}^{t} \d r_n \int_0^{r_n} \d r_{n-1} \cdots \int_{0}^{r_2} \d r_1 \prod_{i=0}^n \frac1{(r_{i+1} -r_{i})^\delta } = \frac{\Gamma(1-\delta)^{n+1}} {\Gamma((n+1) (1-\delta))} t^{n(1-\delta)- \delta} .$$
\end{lemma}

\begin{proof}
First we prove
  \begin{equation}\label{beta-equality}
  \int_{0}^{t} \d r_n \int_0^{r_n} \d r_{n-1} \cdots \int_{0}^{r_2} \d r_1 \prod_{i=1}^n \frac1{(r_{i+1} -r_i)^\delta } = t^{n(1-\delta)} \prod_{i=1}^n B\big(1-\delta, 1+(i-1)(1-\delta) \big),
  \end{equation}
where $B(\alpha, \beta)$ is the Beta function:
  $$B(\alpha, \beta) = \int_0^1 \theta^{\alpha -1} (1-\theta)^{\beta -1}\,\d\theta, \quad \alpha, \beta>0. $$
We proceed by induction. For $n=1$, noting that $r_2=t$, we change the variable $\theta= r_1/t$ and get
  $$\int_0^t \frac{\d r_1}{(t -r_1)^\delta} = t^{1-\delta} \int_0^1 \frac{\d \theta}{(1 -\theta)^\delta} = t^{1-\delta} \int_0^1 \theta^0 (1-\theta)^{-\delta} \,\d \theta = t^{1-\delta} B(1-\delta, 1). $$
Therefore the equality holds when $n=1$. Now suppose the equality holds for $n-1$, we prove it for $n$. By the induction hypothesis,
  $$\int_0^{r_n} \d r_{n-1} \cdots \int_{0}^{r_2} \d r_1 \prod_{i=1}^{n-1} \frac1{(r_{i+1} -r_i)^\delta} = r_n^{(n-1)(1-\delta)} \prod_{i=1}^{n-1} B\big(1-\delta, 1+(i-1)(1-\delta) \big),$$
thus, noticing that $r_{n+1} =t$,
  $$\int_{0}^{t} \d r_n \int_0^{r_n} \d r_{n-1} \cdots \int_{0}^{r_2} \d r_1 \prod_{i=1}^n \frac1{(r_{i+1} -r_i)^\delta} = \prod_{i=1}^{n-1} B\big(1-\delta, 1+(i-1)(1-\delta) \big) \int_0^t \frac{r_n^{(n-1)(1-\delta)}}{(t-r_n)^\delta}\,\d r_n .$$
We have, by changing variable $\theta = r_n/t$,
  $$\int_0^t \frac{r_n^{(n-1)(1-\delta)}}{(t-r_n)^\delta}\,\d r_n = t^{n(1-\delta)} \int_0^1 \theta^{(n-1)(1-\delta)} (1-\theta)^{-\delta} \,\d\theta = t^{n(1-\delta)} B\big(1-\delta, 1+(n-1)(1-\delta) \big) . $$
Substituting this result into the previous one gives us the identity \eqref{beta-equality}.

Next, it is well known that
  $$B(\alpha, \beta) = \frac{\Gamma(\alpha) \Gamma(\beta)}{\Gamma(\alpha+\beta)}.$$
Therefore,
  $$\prod_{i=1}^n B\big(1-\delta, 1+(i-1)(1-\delta) \big) =  \prod_{i=1}^n \frac{\Gamma(1-\delta) \Gamma(1+(i-1)(1-\delta))} {\Gamma(1+i(1-\delta))} = \frac{\Gamma(1-\delta)^n}{\Gamma(1+n(1-\delta)} . $$
Combining this with \eqref{beta-equality} we obtain the desired formula.

The proof of the second identity is similar, by first establishing the identity
  $$\int_{0}^{t} \d r_n \int_0^{r_n} \d r_{n-1} \cdots \int_{0}^{r_2} \d r_1 \prod_{i=1}^n \frac1{(r_{i+1} -r_i)^\delta } = t^{n(1-\delta) -\delta} \prod_{i=1}^n B\big(1-\delta, i(1-\delta) \big).$$
We omit the details here.
\end{proof}

As a consequence, we have the following estimates.

\begin{corollary}\label{cor-estimate}
Under the Hypotheses \ref{hypothe} and \ref{hypothe-1}, for any $n\geq 0$ and $t>0$,
 \begin{equation}\label{eq:cor-estimate}
 \|v^{n} (t)\|_\infty \leq \|u_0\|_\infty \|B\|_\infty^n C_\delta^n t^{n(1-\delta)} \frac{\Gamma(1-\delta)^n} {\Gamma(1+n(1-\delta))}
 \end{equation}
and
  $$\|D v^n(t)\|_\infty \leq \|u_0\|_\infty \|B\|_\infty^n C_\delta^{n+1} t^{n(1-\delta) -\delta} \frac{\Gamma(1-\delta)^{n+1}} {\Gamma((n+1) (1-\delta))}.$$
\end{corollary}

\begin{proof}
The case $n=0$ follows directly from \eqref{derivative-formula.2}. Combining Lemma \ref{lem-technical} and Corollary \ref{cor-iteration-new}, we obtain the general cases.
\end{proof}

Now we can prove the existence of limit for the iteration scheme \eqref{new-iteration}.

\begin{proposition}\label{prop-series}
Assume the Hypotheses \ref{hypothe} and \ref{hypothe-1}. For any $T>0$, the series
  $$\sum_{n=0}^\infty v^n(t,x) $$
converge uniformly on $[0,T]\times H$. Moreover, for any $t_0\in (0,T)$, the series
  $$\sum_{n=0}^\infty D v^n(t,x)$$
converge uniformly on $[t_0, T]\times H$.
\end{proposition}

\begin{proof}
We only prove the first assertion; the proof of the second one is similar. By Corollary \ref{cor-estimate} and using the ratio test, it is sufficient to show that
  $$\lim_{n\to \infty} \frac{\Gamma(1+n(1-\delta))}{\Gamma(1+(n+1)(1-\delta))}=0. $$
This follows from elementary calculations. Indeed, setting $\alpha =1-\delta$ for simplicity of notation,
  $$\frac{\Gamma(1+n\alpha)}{\Gamma(1+(n+1)\alpha)} = \frac{n\alpha}{(n+1)\alpha} \cdot \frac{n\alpha-1}{(n+1)\alpha -1} \cdots \frac{1+(n\alpha)}{1+\alpha +(n\alpha)} \cdot \frac{\Gamma((n\alpha))}{\Gamma(\alpha+ (n\alpha))}, $$
where $(n\alpha)$ is the decimal part of $n\alpha$. Using the simple inequality $\log(1+x) <x$ for all $x\in (-1,0)$, we have
  $$\log\bigg(\frac{n\alpha-k}{(n+1)\alpha -k} \bigg)= \log\bigg(1- \frac{\alpha}{(n+1)\alpha -k} \bigg) < -\frac{\alpha}{(n+1)\alpha -k}.$$
Hence,
  $$\log \frac{\Gamma(1+n\alpha)}{\Gamma(1+(n+1)\alpha)} < -\alpha \bigg(\frac{1}{(n+1)\alpha} + \frac{1}{(n+1)\alpha -1} + \cdots + \frac{1}{1+\alpha +(n\alpha)} \bigg) +\log\frac{\Gamma((n\alpha))}{\Gamma(\alpha+ (n\alpha))} . $$
Note that the first part on the right hand side tends to $-\infty$ as $n\to \infty$, while the last part is uniformly bounded in $n$, thus we conclude the result.
\end{proof}

Thanks to Proposition \ref{prop-series}, we can define the limit
  $$u(t,x) = \lim_{n\to \infty} u^n(t,x)= \lim_{n\to \infty} \sum_{i=0}^n v^i(t,x);$$
moreover, for any $t>0$, one has $u(t)\in C_b^1(H)$ and
  $$D u(t,x) = \lim_{n\to \infty} D u^n(t,x) =\lim_{n\to \infty} \sum_{i=0}^n Dv^i(t,x)$$
which holds uniformly on $[t_0, T]\times H$ for any $0<t_0 <T$. Finally we can prove the main result.

\begin{theorem}\label{thm-main-result}
The limit $u(t,x)$ is the unique solution to the Kolmogorov equation \eqref{KolE} in the following sense:
\begin{itemize}
\item[\rm (a)] for any $T>0$, $u(t,x)$ is uniformly bounded for $(t,x)\in [0,T]\times H$, and $u(t)\in C_b^1(H)$ for any $t>0$;
\item[\rm (b)] for any $T>0$, one has $\int_0^T \|D u(t)\|_\infty \,\d t<\infty$;
\item[\rm (c)] it satisfies the mild formulation \eqref{mild-sol} for any $t>0$ and $x\in H$.
\end{itemize}
\end{theorem}

\begin{proof}
Obviously our limit verifies (a). Next,
  $$\|D u(t)\|_\infty \leq \sum_{n=0}^\infty \big\| D v^{n}(t) \big\|_\infty \leq \|u_0\|_\infty \sum_{n=0}^\infty \|B\|_\infty^n C_\delta^{n+1} t^{n(1-\delta) -\delta} \frac{\Gamma(1-\delta)^{n+1}} {\Gamma((n+1) (1-\delta))}.$$
Therefore,
  \begin{equation}\label{thm-main-result.0}
  \aligned
  \int_0^T \|D u(t)\|_\infty \,\d t &\leq \|u_0\|_\infty \sum_{n=0}^\infty \|B\|_\infty^n C_\delta^{n+1}  \frac{ \Gamma(1- \delta)^{n+1}} {\Gamma((n+1) (1-\delta))} \int_0^T t^{n(1-\delta) -\delta} \,\d t \\
  &= \|u_0\|_\infty \sum_{n=0}^\infty \|B\|_\infty^n C_\delta^{n+1}  \frac{ \Gamma(1- \delta)^{n+1}} {\Gamma((n+1) (1-\delta))} \frac{T^{(n+1) (1-\delta)}}{(n+1) (1-\delta)},
  \endaligned
  \end{equation}
which shows that (b) is also satisfied. Moreover, for any $t>0$ and $x\in H$,
  $$\aligned
  \bigg|\int_0^t \big( S_{t-s} \<B, D u(s) \>\big)(x)\,\d s\bigg| \leq \int_0^t \|\<B, D u(s) \>\|_\infty \,\d s \leq \|B\|_\infty \int_0^t \|D u(s)\|_\infty \,\d s.
  \endaligned$$
This implies the integral in the signs of absolute value makes sense.

It remains to check that $u(t,x)$ verify \eqref{mild-sol}. By the iteration scheme \eqref{new-iteration}, one has, for any $n>1$,
  \begin{equation}\label{thm-main-result.1}
  u^n(t,x) = u^0(t,x) + \int_0^t \big( S_{t-s} \big\<B, D u^{n-1}(s) \big\>\big)(x)\,\d s \quad \mbox{for all } t>0,\, x\in H.
  \end{equation}
The left hand side converges uniformly to $u(t,x)$ on $[0,T]\times H$ for any $T>0$. It suffices to show the uniform convergence of the right hand side. We have
  $$\aligned
  &\ \bigg|\int_0^t \big( S_{t-s} \big\<B, D u^{n-1}(s) \big\>\big)(x)\,\d s - \int_0^t \big( S_{t-s} \<B, D u(s) \>\big)(x)\,\d s\bigg| \\
  \leq &\ \int_0^t \big\| \big\<B, D u^{n-1}(s)- D u(s) \big\> \big\|_\infty \,\d s\leq \|B\|_\infty \sum_{i=n}^\infty \int_0^t \big\| D v^{i}(s) \big\|_\infty\,\d s.
  \endaligned $$
Similarly to the calculations in \eqref{thm-main-result.0}, we can show that the right hand side vanishes as $n$ goes to infinity. Therefore we let $n\to \infty$ on both sides of \eqref{thm-main-result.1} and conclude that $u(t,x)$ satisfies \eqref{mild-sol} uniformly in $(t,x)\in [0,T]\times H$.

Finally we prove the uniqueness of solutions. Suppose $u(t,x)$ and $\tilde u(t,x)$ are two solutions to \eqref{KolE}  with the properties (a)--(c). Then, for any $t>0$ and $x\in H$,
  $$u(t,x) -\tilde u(t,x)= \int_0^t S_{t-s}\big(\<B, D(u(s) -\tilde u(s))\>\big)(x) \,\d s. $$
Therefore,
  \begin{equation}\label{thm-main-result.2}
  \aligned
  |u(t,x) -\tilde u(t,x)| &\leq \|B\|_\infty \int_0^t \|D u(s) - D\tilde u(s) \|_\infty\,\d s.
  \endaligned
  \end{equation}
Moreover, by Proposition \ref{derivative-formula},
  $$\aligned
  |D(u(t,x) -\tilde u(t,x))| &\leq \int_0^t \big| D S_{t-s}\big(\<B, D(u(s) -\tilde u(s))\>\big)(x)\big| \,\d s\\
  &\leq \int_0^t \big\| \<B, D(u(s) -\tilde u(s))\>\big\|_\infty \|\Lambda(t-s)\|_{\mathcal L(H)} \,\d s \\
  &\leq \|B\|_\infty \int_0^t \|D u(s) - D\tilde u(s) \|_\infty \|\Lambda(t-s)\|_{\mathcal L(H)} \,\d s.
  \endaligned $$
Hence,
  $$\aligned
  \int_0^t \|D u(s) - D\tilde u(s) \|_\infty \,\d s &\leq \|B\|_\infty \int_0^t\! \int_0^s \|D u(r) - D\tilde u(r) \|_\infty \|\Lambda(s-r)\|_{\mathcal L(H)} \,\d r\d s\\
  &= \|B\|_\infty \int_0^t \|D u(r) - D\tilde u(r) \|_\infty\, \d r \int_r^t \|\Lambda(s-r)\|_{\mathcal L(H)} \,\d s \\
  &\leq \bigg[\|B\|_\infty \int_0^t \|\Lambda(s)\|_{\mathcal L(H)} \,\d s \bigg]\! \int_0^t \|D u(r) - D\tilde u(r) \|_\infty\, \d r .
  \endaligned $$
Under Hypothesis \ref{hypothe}-(iv), there is some $t_1>0$ such that $\|B\|_\infty \int_0^{t_1} \|\Lambda(s)\|_{\mathcal L(H)} \,\d s <1$. Then
  $$\int_0^t \|D u(s) - D\tilde u(s) \|_\infty \,\d s =0 \quad \mbox{for all } t\leq t_1.$$
Combining this with \eqref{thm-main-result.2} we see that $u(t,x) = \tilde u(t,x)$ for any $(t,x)\in [0,t_1] \times H$. Next, by the semigroup property, it is easy to show that, for $t\in (0,t_1]$,
  $$\aligned
  u(t+t_1, x)= S_{t} u_{t_1}(x) + \int_0^t S_{t-s}\big(\<B, D u(t_1+ s)\> \big)(x) \,\d s.
  \endaligned$$
Repeating the above procedure we can prove the uniqueness on the interval $[t_1, 2t_1]$ and so on. Thus we complete the proof.
\end{proof}

\section{Numerical Simulations}
In this section we propose some experiment of the iteration scheme \eqref{new-iteration} studied in Section \ref{sec-theory} in the finite dimensional setting.
We have in mind the framework of Example \ref{ex:laplacian}, i.e. $A = \Delta$. Since we are in the finite dimensional setting this choice corresponds to take $A \in \R^{d}\otimes \R^{d}$  as the diagonal matrix where $A_{k,k} = -k^{2},\, k=1,\ldots, d$. Moreover we consider the matrix $ Q = \sigma^{2} I_{d\times d}$ where $I_{d\times d}$ is the identity matrix over $\R^{d}$, and the parameter $\sigma$ will be specified below (see Table \ref{tab:modelparameters} for reference parameters).

We will consider two main classes of examples as a benchmark for our approximation scheme.
First, we consider the nonlinear vector field
\begin{equation}\label{eq:sin}
B(x)_{i} = \sin(x_{i}),\quad i=1,\dots, d,
\end{equation}
i.e. we apply the sine function to all the components. This nonlinearity will be the easier one of our examples since it is close to linear, at least for small values of $x$. We will also consider some variation of the previous example, made by
\begin{equation}\label{eq:sinSkew}
B(x)_{i} = \sin(x_{i})(B_{m}x)_{i},\quad i=1,\dots, d,
\end{equation}
where $B_{m} \in \R^{d}\times \R^{d}$ is the skew symmetric matrix
\[
(B_{m})_{i,j} = \begin{cases}
1	&\text{ if } i<j;\\
-1 	&\text{ if } i>j;\\
0	&\text{ if } i=j,
\end{cases}
\]
i.e. the Toeplitz matrix with all one above the diagonal and minus one below.
This example is more complex than the previous one. It is significant since it deals with skew symmetric matrices, inducing rotations, which are a first simple step in the direction of fluid dynamics. The vector field $B_{m}x$ is also multiplied by the function $\sin(x)$ in order to make example \eqref{eq:sinSkew} nonlinear. Notice that this last example is not covered by our present theory, since it does not satisfy Hypothesis \ref{hypothe-1}. However, even if \eqref{eq:sinSkew} is not bounded, it satisfies a linear growth condition. We hope to improve our theory and the generality of the assumptions in such direction in a future research, and limit ourselves to some numerical experiments for the present work.

\begin{figure}[]
\centering
\includegraphics[width=\textwidth]{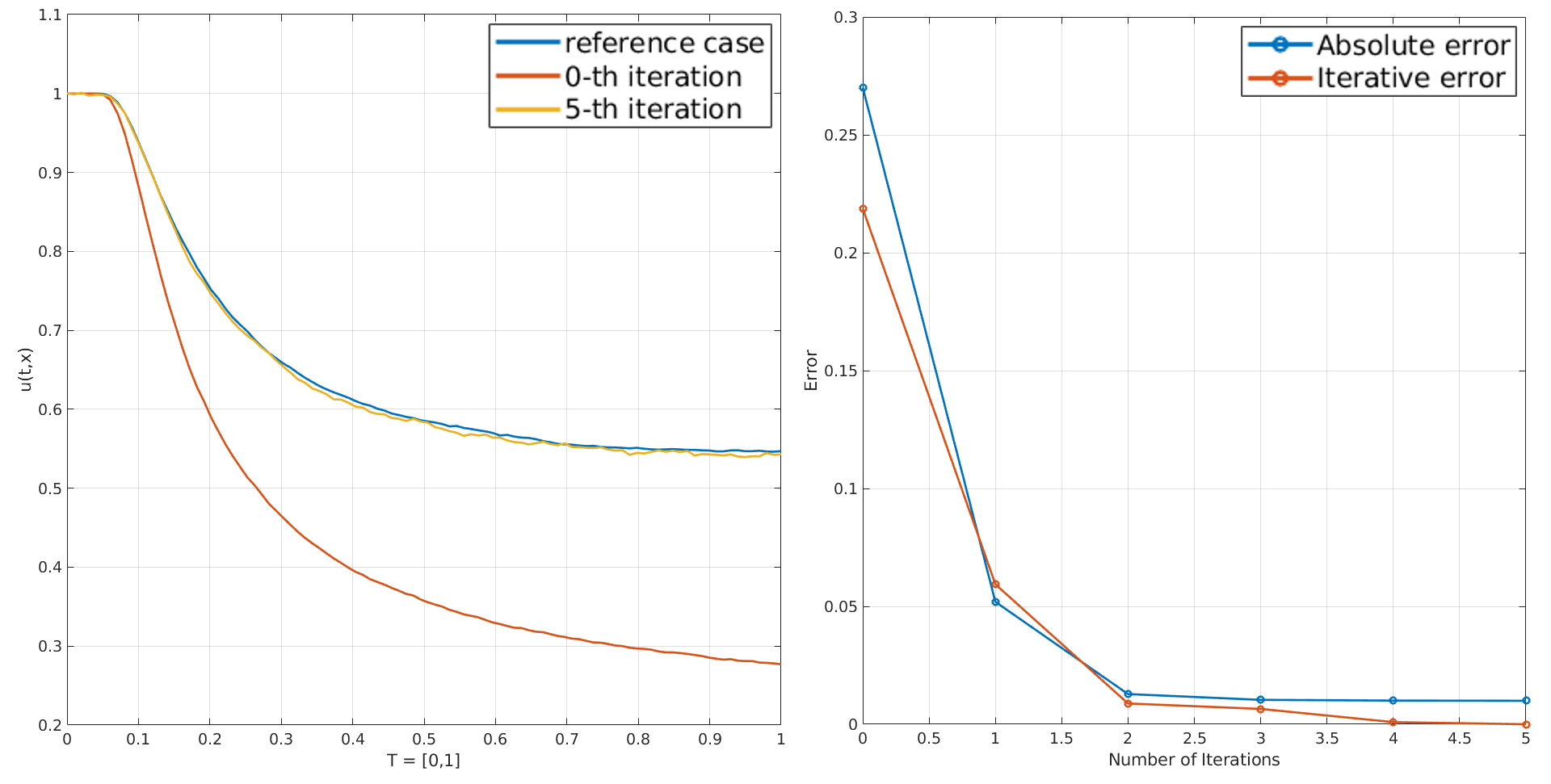}
\caption{Sine case \eqref{eq:sin}. Left: trajectory of $u(t,x)$ for $t \in [0,T],\, d = 10$. Right: difference between consecutive iterations and error with respect to the reference case, X-axis number of iterations.}
\label{fig:sinFull}
\end{figure}

Second, we consider the following class of polynomial nonlinearities
\begin{equation}\label{eq:polynomialbounded}
B(x)_{i} = \norm{\overline{y}}\frac{(\overline{y}_{i}-x_{i})\abs{\overline{y}_{i}-x_{i}}^{p-1}}{\norm{\overline{y}}+\norm{\overline{y}-x}^{p}},\quad i=1,\dots,d
\end{equation}
where $\overline{y} \in \R^{d}$ is fixed. Note that this example appeals to the one dimensional case
\[
B(x) = (\overline{y}-x)\abs{\overline{y}-x}^{p-1},
\]
for which the dynamical system
\[
\dot{x}(t) = B(x(t))
\]
has the singleton $\{\overline{y}\}$ as a stable attractor. The reason behind the example \eqref{eq:polynomialbounded} is the following:
it is close to a polynomial nonlinearity, so that it makes a significant test case; at the same time, the normalization by the factor $\norm{\overline{y}}/(\norm{\overline{y}}+\norm{\overline{y}-x}^{p})$ makes it a bounded operator, so that it fulfills Hypothesis \ref{hypothe-1}.

In all the examples above we adopt the following choice of initial condition
\[
u_{0}(x) = \mathds{1}_{\{\norm{x}\geq H\}},
\]
where the parameter $H$ is set to $1$ (see Table \ref{tab:modelparameters}).

\subsection{Approximation schemes}

\paragraph*{Standard Monte Carlo approach.} Since an explicit solution for Equation \eqref{KolEq} is not available we will always compare to the solution obtained by means of Monte Carlo simulation of the nonlinear process $X^{x}_{t}$:
\begin{equation}\label{eq:reference case}
u (  t,x )  =\mathbb{E}\left[  u_{0} (  X_{t}^{x} )
\right] \simeq \frac{1}{N_{s}} \sum_{i=1}^{N_{s}} u_{0}\big(X^{x,i}_{t} \big),
\end{equation}
where $N_{s}$ is the number of samples considered, and the processes $X^{x,i}_{t}, i=1,\dots,N_{s}$ are independent copies of $X^{x}_{t}$. To compute samples of the process $X^{x,i}_{t}$ we use the Euler-Maruyama scheme with a very fine time step in order to get a good approximation to be used as a comparison. The solution computed by \eqref{eq:reference case} will always be referred to in what follows as the reference case.

\begin{figure}[]
\centering
\includegraphics[width=\textwidth]{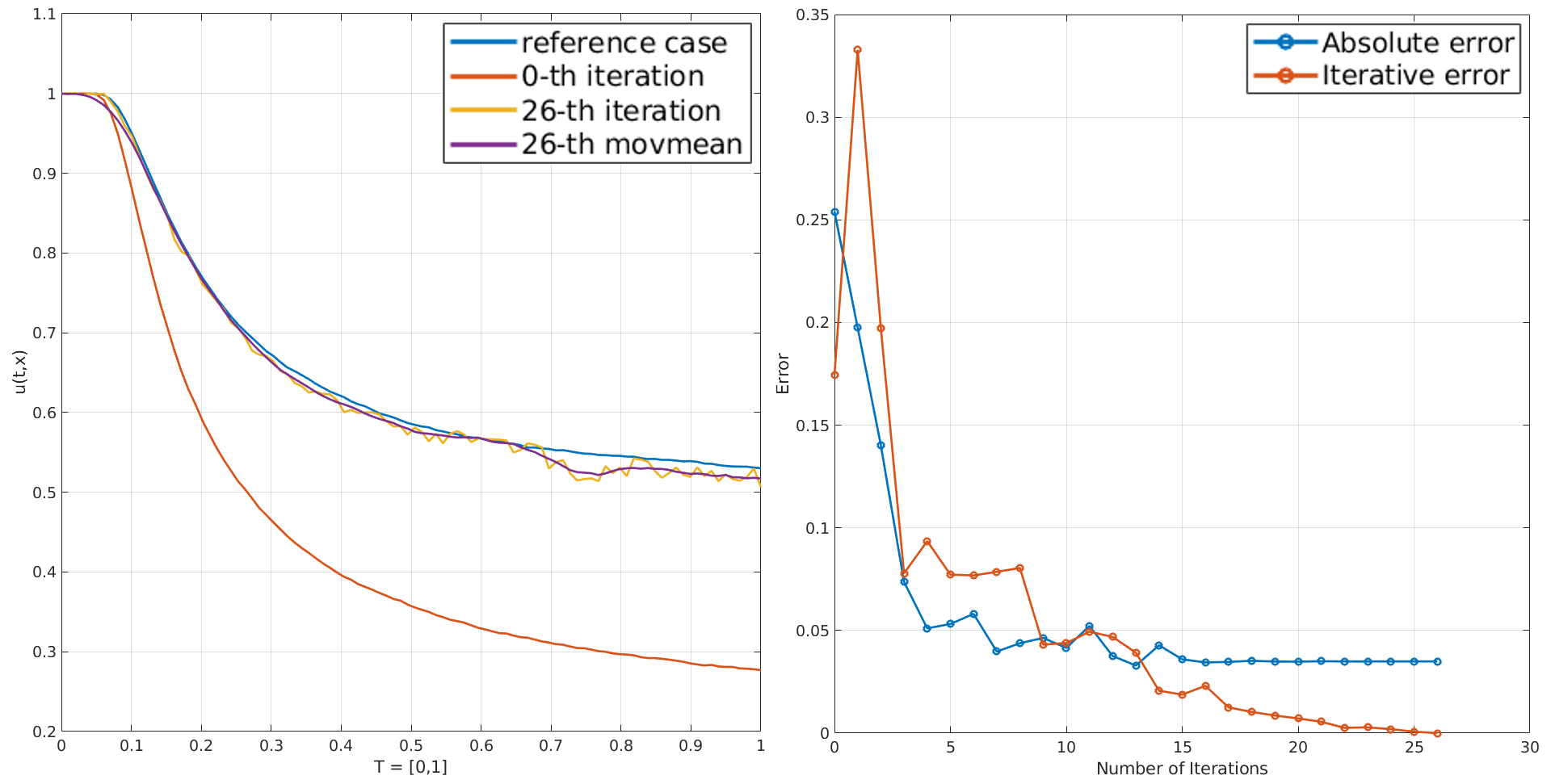}
\caption{Polynomial bounded quadratic case \eqref{eq:polynomialbounded} with $p=2$, $d = 10$. Left: trajectory of $u(\cdot,x)$ for $t \in [0,T]$. The purple line is obtained by applying a moving average smoothing to the last iteration. Right: difference between consecutive iterations and error with respect to the reference case, X-axis number of iterations.}
\label{fig:quadraticfull}
\end{figure}

\paragraph*{Numerical iteration scheme.} Under our assumption, since $A$ and $Q$ are diagonal, we can rewrite the equations for the processes $Z^{x}_{t}$ and $Z_{t}$ in a simple way: for $k=1,\ldots,d$,
\begin{equation}\label{eq:Znum}
\left\{ \aligned
\d Z^{k}_{t} &= -k^{2} Z^{k}_{t}\,\d t + \d W^{k}_{t},	\\
Z^{k}_{0} &= 0
\endaligned \right.
\end{equation}
and
\[
Z^{x,k}_{t} = e^{-k^{2}t}x_{k} + \sigma Z^{k}_{t}.
\]
We remark that, differently from $Z_{t}$, the process $Z^{x}_{t}$ depends also on the parameter $\sigma$, but we do not explicitly write $Z^{x,\sigma}_{t}$ for ease of notation. Note that the process $Z_{t}$ depends only on the operators $A$. This opens the possibility of computing $Z^{x}_{t}$, and hence also $u(t,x)$, for many values of $x$ without repeating the computations for $Z_{t}$. The same reasoning holds for different values of $\sigma$, see Figure \ref{fig:longsigma}. Note also that this strategy cannot be applied to the process $X_{t}^{x}$ since in that case the problem is nonlinear.

Once realizations of the process $Z^{x}_{t}$ are computed, we can proceed with the iteration algorithm \eqref{new-iteration}. In order to compute numerically the quantity $v^{n}(t,x)$ appearing in Theorem \ref{thm-iteration} one needs to be able to compute first
\begin{equation}\label{eq:integrand}
\left\langle \Lambda\left(  s\right)  B\left(  Z_{t-s}^{x}\right)
,Q_{s}^{-1/2}\left(  Z_{t}^{x}-e^{sA}Z_{t-s}^{x}\right)  \right\rangle.	
\end{equation}
Since $A$ and $Q$ are diagonal and explicit (see the beginning of this section), one has
\[
\left(  Q_{t}\right)  _{k,k}=\int_{0}^{t}\left(  e^{sA}\right)  _{k,k}%
Q_{k,k}\big(  e^{sA^{\ast}}\big)_{k,k}\,\d s=\int_{0}^{t}e^{-2sk^{2}%
}\sigma^{2}\,\d s=\frac{\sigma^{2}}{2k^{2}}\big(  1-e^{-2tk^{2}}\big),
\]%
\[
\big(  Q_{t}^{-1/2}\big)  _{k,k}=\frac{\sqrt{2}k}{\sigma%
\sqrt{1-e^{-2tk^{2}}}}, \quad
\left(  \Lambda\left(  t\right)  \right)  _{k,k}=\frac{\sqrt{2%
}ke^{-tk^{2}}}{\sigma\sqrt{1-e^{-2tk^{2}}}},
\]
and thus, 
\begin{align*}
&\, \left\langle \Lambda\left(  s\right)  B\left(  Z_{t-s}^{x}\right)
,Q_{s}^{-1/2}\left(  Z_{t}^{x}-e^{sA}Z_{t-s}^{x}\right)  \right\rangle \\
=&\, \sum_{k=1}^{d} \frac{2k^{2}e^{-sk^{2}}}{\sigma^{2}(1-e^{-2sk^{2}})}B\left(
Z_{t-s}^{x}\right)_{k}  \left(  Z_{t}%
^{x,k}-e^{-sk^{2}}Z_{t-s}^{x,k}\right) .
\end{align*}

\begin{figure}[]
\begin{subfigure}{0.49\textwidth}
\includegraphics[width=\textwidth]{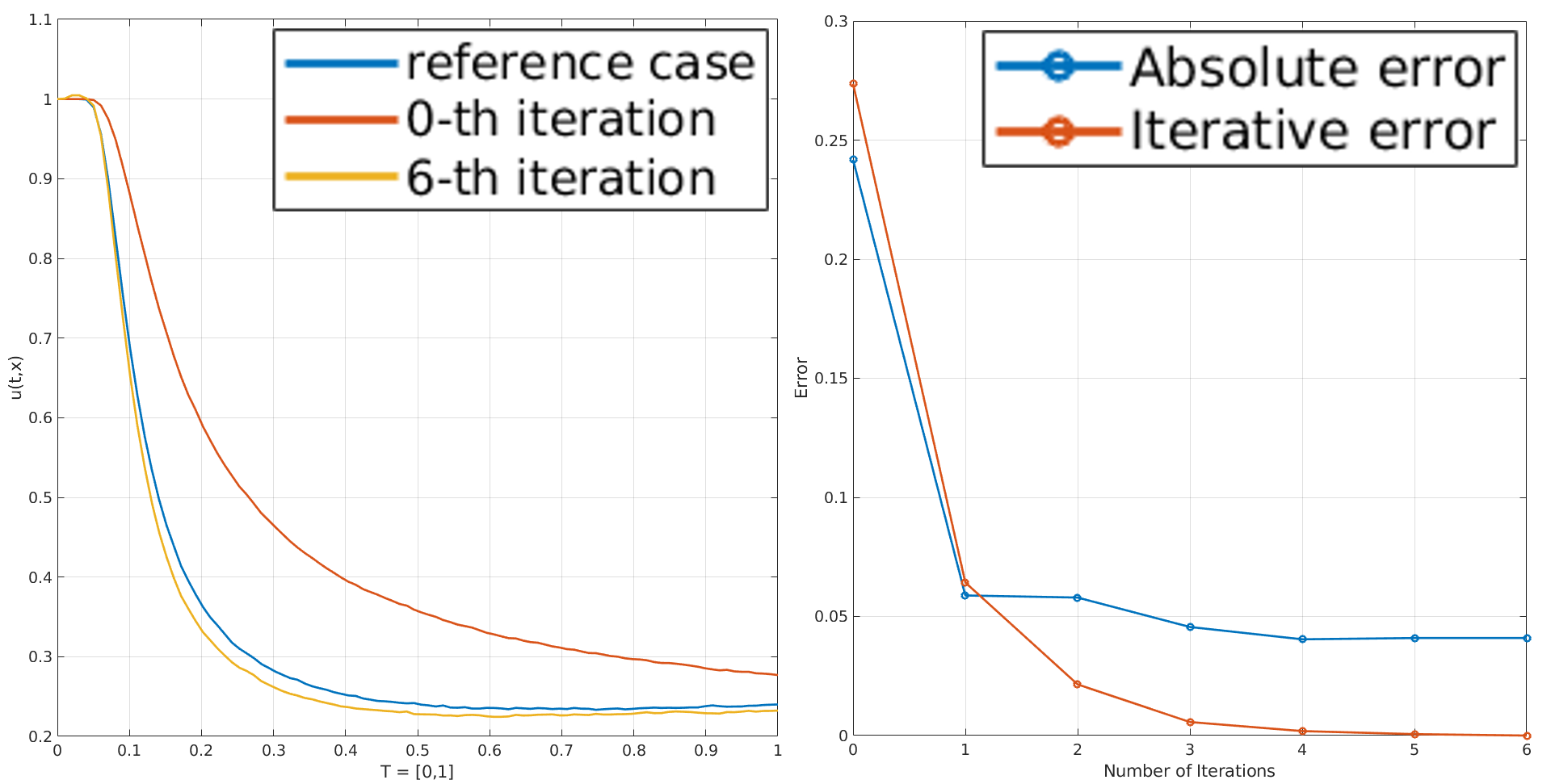}
\label{fig:sinBmFull}
\end{subfigure}
\begin{subfigure}{0.49\textwidth}
\includegraphics[width=\textwidth]{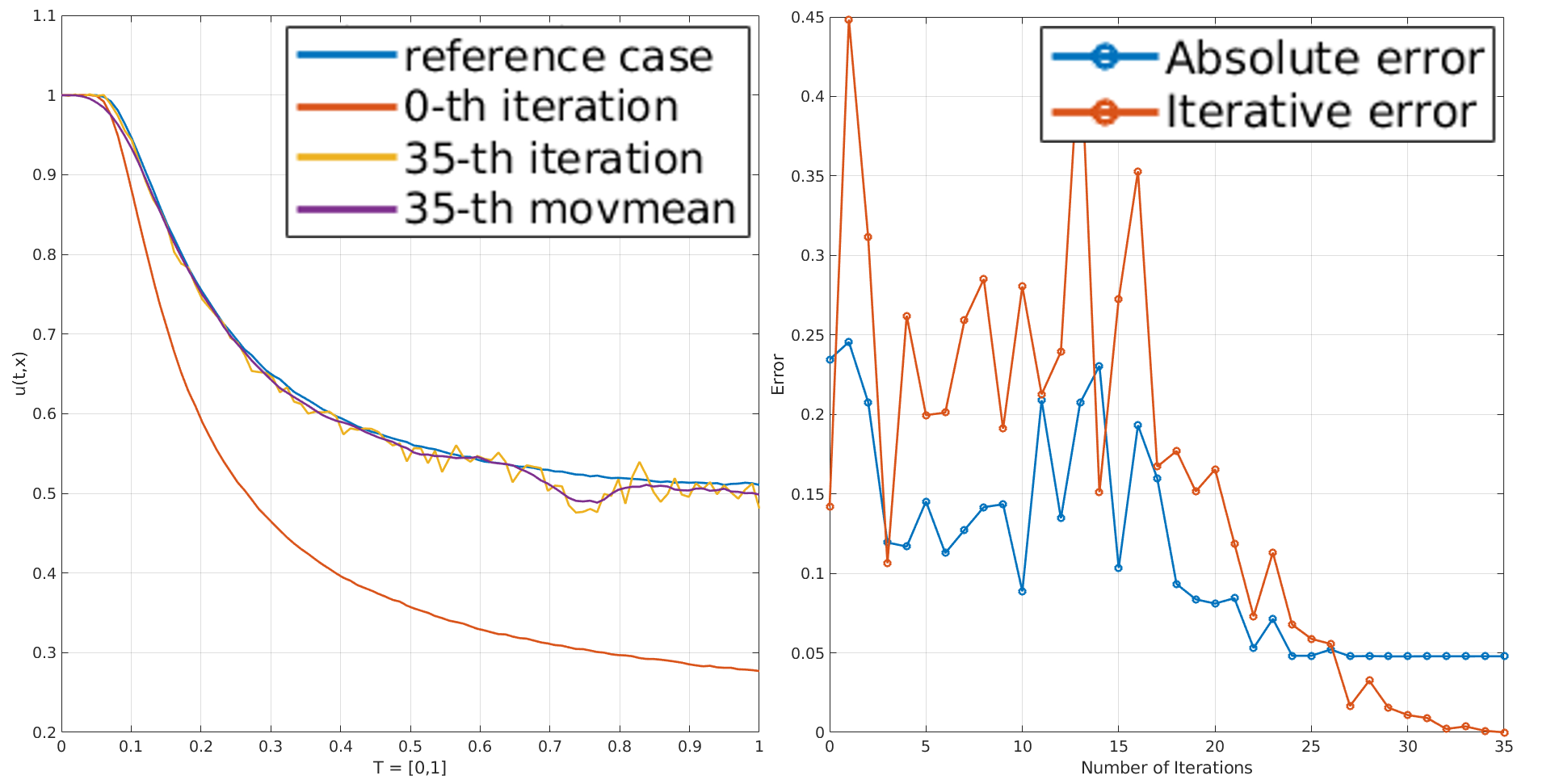}
\label{fig:cubicFullBad}
\end{subfigure}
\caption{Left block: Sine times skew-symmetric case \eqref{eq:sinSkew} with $d = 10$. Right block: Polynomial bounded cubic case \eqref{eq:polynomialbounded} $p=3$, $d = 10$. The purple line is obtained by applying a moving average smoothing to the last iteration.}
\label{fig:sinBmFull+cubicFullBad}
\end{figure}
Hence, when integrating expression \eqref{eq:integrand}, by change of variable we have
\begin{align}
&\, \int_{0}^{t}\left\langle \Lambda\left(  s\right)  B\left(  Z_{t-s}^{x}\right)
,Q_{s}^{-1/2}\left(  Z_{t}^{x}-e^{sA}Z_{t-s}^{x}\right)  \right\rangle\,\d s \nonumber  \\
=&\, \int_{0}^{t}\left\langle \Lambda\left(  t-s\right)  B\left(  Z_{s}^{x}\right)
,Q_{t-s}^{-1/2}\left(  Z_{t}^{x}-e^{(t-s)A}Z_{s}^{x}\right)  \right\rangle \,\d s \nonumber  \\
=&\, \int_{0}^{t}\sum_{k=1}^{d} \frac{2k^{2}e^{-(t-s)k^{2}}}{\sigma^{2}(1-e^{-2(t-s)k^{2}})}B\left(
Z_{s}^{x}\right)_{k}  \left(  Z_{t}%
^{x,k}-e^{-(t-s)k^{2}}Z_{s}^{x,k}\right)\,\d s.  \label{eq:singleint}
\end{align}
Changing variable provides a significant advantage when performing numerical integration. In fact it is more complex to compute $Z_{t-s}^{x}$ than $\Lambda(t-s)$ \big(resp. $Q_{t-s}^{1/2}$\big) since $Z^{x}$ is random and hence we would have been obliged to reverse the time for every sample of the process. On the other hand the matrix $\Lambda(t-s)$ \big(resp. $Q_{t-s}^{1/2}$\big) is deterministic so that changing time $s\mapsto t-s$ can be done only once.

Moreover, thanks to Corollary \ref{cor-iteration-new}, it is possible to compute $v^{n}(t,x)$ with a single time integration from the previous step.
Introduce
\[
I^{n}(t,x) = \int_{0}^{t} \d r_n \int_0^{r_n} \d r_{n-1} \cdots \int_{0}^{r_2} \d r_1
   \ \prod_{i=1}^n \Big\<\Lambda(r_{i+1}- r_{i}) B\big(Z_{r_i}^x \big), Q_{r_{i+1}- r_{i}}^{-1/2} \big( Z_{r_{i+1}}^x - e^{(r_{i+1}- r_{i}) A} Z_{r_i}^x \big) \Big\>
\]
and notice that, due to Equation \eqref{iteration-new}, we have
\[
v^{n}(t,x) = \E\big[u_0( Z_t^x) I^{n}(t,x) \big].
\]
Since
\[
I^{n+1}(t,x) = \int_{0}^{t}\left\langle \Lambda\left(  t-s\right)  B\left(  Z_{s}^{x}\right)
,Q_{t-s}^{-1/2}\left(  Z_{t}^{x}-e^{(t-s)A}Z_{s}^{x}\right)  \right\rangle I^{n}(s,x)\,\d s,
\]
once we have computed $I^{n}$, computing $I^{n+1}$ is a matter of a single integration. This is really crucial because, otherwise, by using the direct expression \eqref{iteration} in Theorem \ref{thm-iteration}, to compute $v^{n}(t,x)$ one should have done an $n$-dimensional numerical integration, independently on the previous iteration.

\paragraph*{Stopping conditions.}
Since the numerical scheme is iterative and since an exact solution is not available, we adopt a consecutive-iterations stopping condition. At every step we measure the difference between consecutive iterations and stop when this difference is below a certain threshold $tol$. Specifically we adopt two strategies in different situations: when we compute the entire trajectory of $u(t,x)$ for $t \in [0,T]$, we measure
\[
err(n) := \sup_{t\in [0,T]}\norm{v^{n}(t,x)}
\]
and stop the iterations if $err(n) < tol$ (see Figures \ref{fig:sinFull} and \ref{fig:quadraticfull}); when we are interested only in $u(T,x)$ for a fixed $T$, then
\[
err(n) := \abs{v^{n}(T,x)}
\]
and adopt the same stopping rule (Figures \ref{fig:longsigma} and \ref{fig:sigmaQuadraticBad+changeX}).

The entire procedure  can be summarized in the following scheme:\newline\newline
\begin{algorithm}[H]
\SetAlgoLined
\KwResult{$u^{n}(t,x)$ approximating solution after $n$ iterations}
 Compute $N_{s}$ samples of the process $Z_{t}$\;
 Compute $N_{s}$ samples for $Z^{x}_{t}$ starting from $Z_{t}$\;
 Compute $u^{0}(t,x)= v^0(t,x)= \E[u_{0}(Z^{x}_t)]$ by Monte Carlo average\;
 Set $err = 1$, $n = 0$\;
 \While{$err > tol$}{
 	Compute $v^{n+1}(t,x)$ as in Corollary \ref{cor-iteration-new}\;
 	Set $u^{n+1}(t,x) = u^{n}(t,x) +  v^{n+1}(t,x)$ \;
 	Set $err = \abs{v^{n+1}(t,x)}$\;
 	Set $n = n+1$\;
} 	
 \caption{Iteration Scheme}
\end{algorithm}

\subsection{Examples}
Here we collect the results obtained, and all the parameters involved in the simulations. Parameters are divided into two categories: those related to the mathematical problem, and those strictly related to the numerical approximations, see Tables \ref{tab:modelparameters} and \ref{tab:numericalparameters}. Those are our reference parameters: we will specify each time any modifications.

In all the figures below, when showing the entire trajectory of the solution $u(t,x)$ for $t \in [0,T]$, we also plot the $0$-th order iteration. This corresponds to the solution of the linear case for \eqref{KolEq}, i.e. the Kolmogorov equation with $B\equiv 0$. This will allow us to compare with the linear case, in order to be sure to have introduced a significant nonlinearity into the problem.

\begin{table}[]
\centering
\begin{minipage}{0.49\textwidth}
\resizebox{\textwidth}{!}{%
\begin{tabular}{rll}
\multicolumn{1}{c}{\textbf{Parameter}} & \multicolumn{1}{c}{\textbf{Value}} & \multicolumn{1}{c}{\textbf{Description}}                \\ \hline
$d$                                    & $10$                               & dimension of the problem
 \\ \hline
$y_0$                                  & $2\mathbf{e}$                      & parameter of the nonlinearity $B$, Polynomial case                      \\ \hline
$x$                                  & $\mathbf{e}$                       & values where the solution $u(t,x)$ is computed        \\ \hline
$\sigma$                               & $1$                                & noise                                                   \\ \hline
$T$                                    & $1$                                & final time of computation for $u(t,x_0)$                \\ \hline
$H$                                    & $1$                                & threshold for the initial condition $u_0(x)$            \\ \hline
\end{tabular}%
}
\caption{Model parameters, $\mathbf{e}$ stands for the vector with all components identically $1$.}
\label{tab:modelparameters}
\end{minipage}
\begin{minipage}{0.49\textwidth}
\resizebox{\textwidth}{!}{%
\begin{tabular}{rll}
\multicolumn{1}{c}{\textbf{Parameter}} & \multicolumn{1}{c}{\textbf{Value}} & \multicolumn{1}{c}{\textbf{Description}}        \\ \hline
$\Delta t$                            & $10^{-4}$                            & time step for Euler scheme                      \\ \hline
$dt$                                  & $10^{-2}$                            & time step for numerical integration             \\ \hline
$N_s$                                 & $10^5$                             & number of samples averages \\ \hline
$tol$                                 & $10^{-3}$                            & tolerance for stopping iterations \\ \hline
\end{tabular}%
}
\caption{Numerical parameters.}
\label{tab:numericalparameters}
\end{minipage}
\end{table}

\paragraph*{Mixed-time-step strategy.} To perform numerical simulation of SDEs and numerical integration we adopt a mixed-time-step strategy. When we compute the reference solution, through the simulation of the process $X^{x}_{t}$, as well as when computing samples of the linear process $Z_{t}$, we adopt a time step $\Delta t$. On the other hand when we perform numerical integration, to compute successive iterations, we adopt a time discretization parameter $dt \gg \Delta t$, see Table \ref{tab:numericalparameters}. This is due to the fact that, in equation
\eqref{eq:X^x_t}, as well as in \eqref{eq:Znum}, a coefficient $-k^{2}$ is present in the $k$-th component of the drift of the equation. This coefficient, and hence the Lipschitz constant of the drift, is growing as the square of the dimension $d$ of the problem. This is caused by the intrinsic exponential decay of equation \eqref{eq:Znum}, which require a high level of precision in computation. Differently, in equation \eqref{eq:singleint}, part of this exponential decay is absorbed by the convolutional structure of the integration. The limits and what is the proper ratio between $\Delta t$ and $dt$ is a difficult topic. A more precise investigation is needed: for the present paper we only highlight the numerical result obtained, and hope to improve the theoretical counterpart in a future work.

\begin{figure}[t]
\begin{subfigure}{0.49\textwidth}
\includegraphics[width=\textwidth]{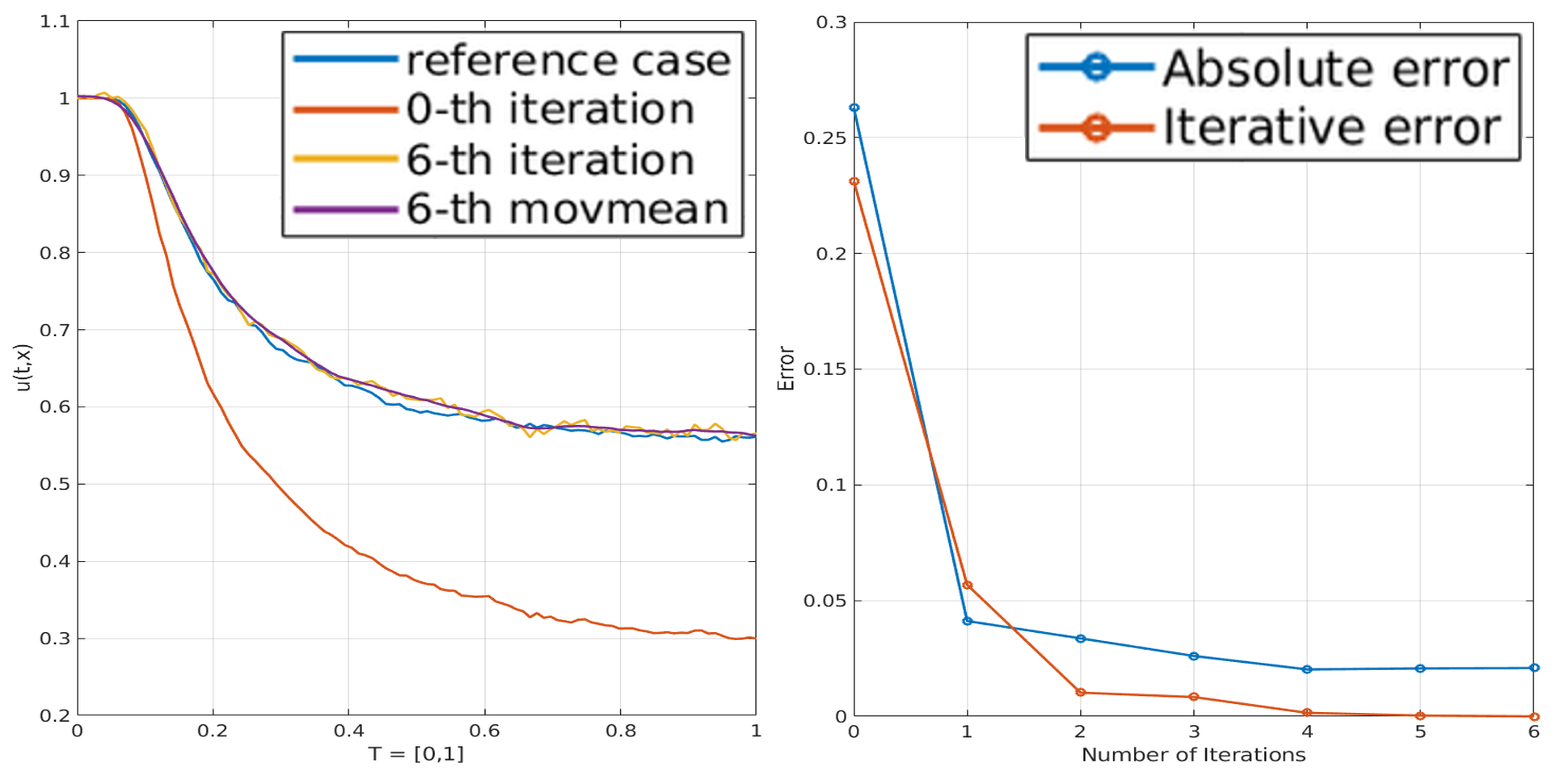}
\label{fig:sinD50}
\end{subfigure}
\begin{subfigure}{0.49\textwidth}
\includegraphics[width=\textwidth]{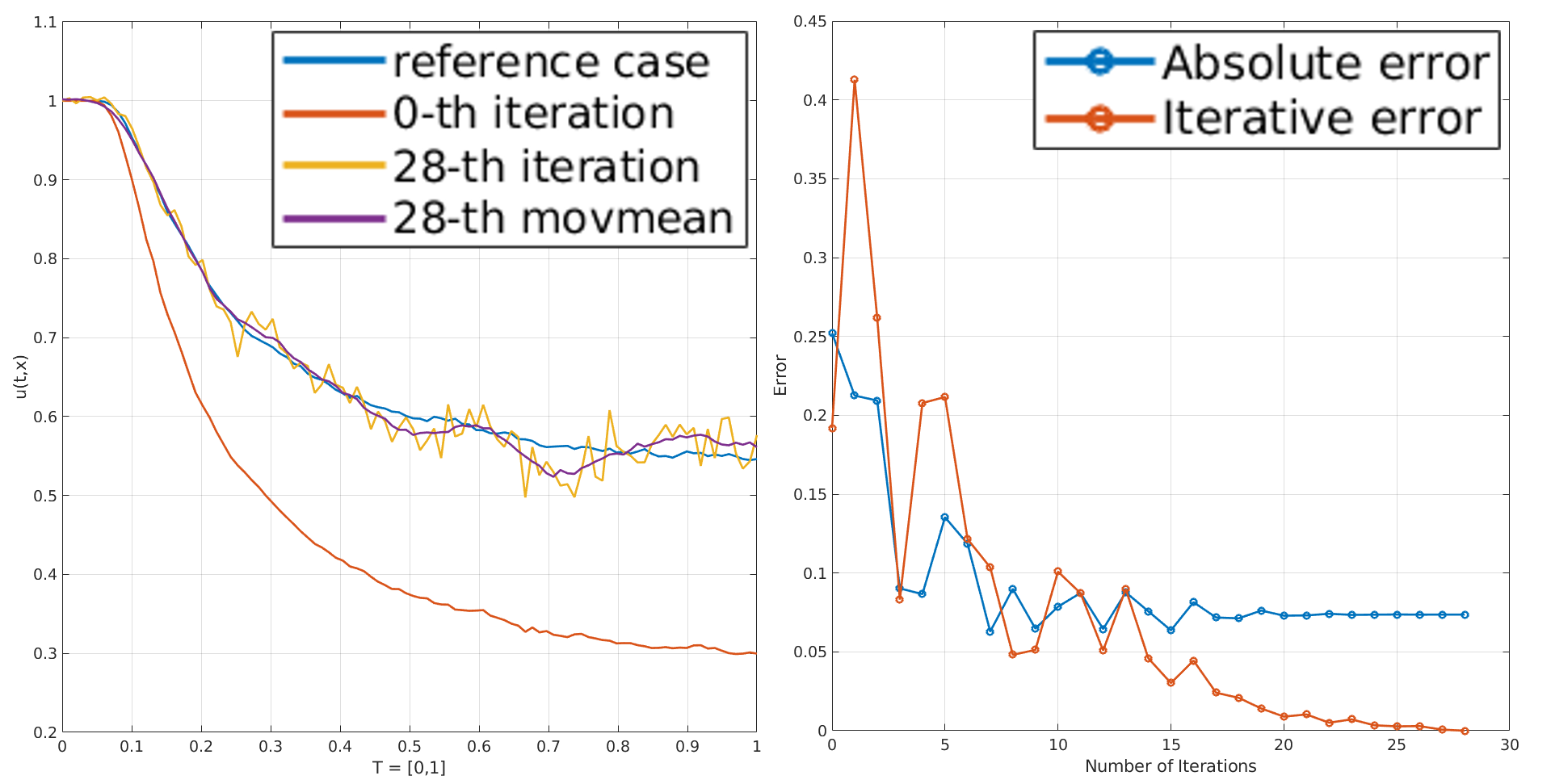}
\label{fig:QuadraticD50}
\end{subfigure}
\caption{Left block: Sine  case \eqref{eq:sin} in dimension $d = 50$, $N_{s} = 10^{4}$. The purple line is obtained by applying a moving-average smoothing to the last iteration. Right block: Polynomial bounded quadratic case \eqref{eq:polynomialbounded} $p=2$ in dimension $d = 50$, $N_{s} = 10^{4}$.}
\label{fig:sinD50+QuadraticD50}
\end{figure}
\begin{figure}[t]
\begin{subfigure}{0.49\textwidth}
\includegraphics[width=\textwidth]{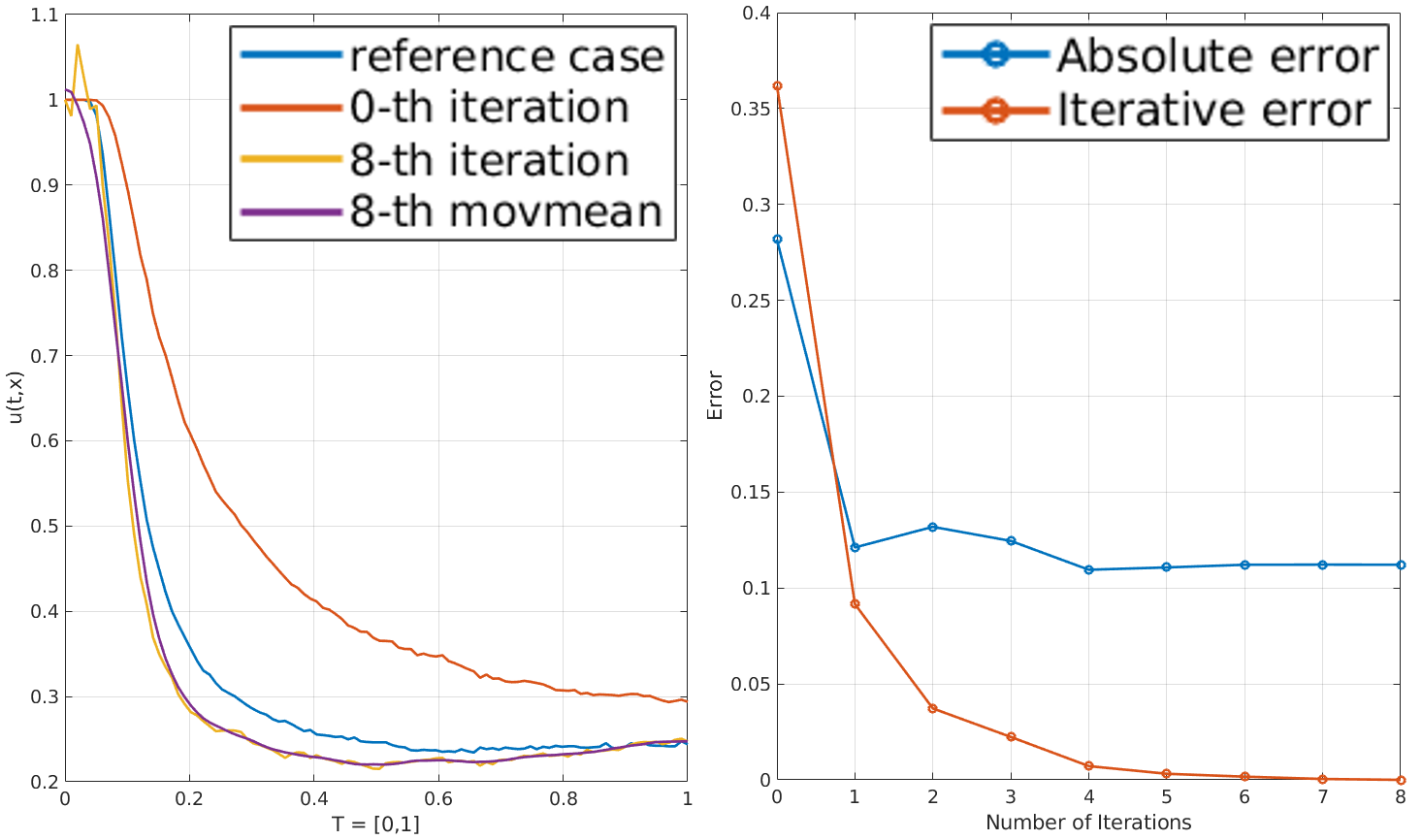}
\label{fig:sinBmD20}
\end{subfigure}
\begin{subfigure}{0.49\textwidth}
\includegraphics[width=\textwidth]{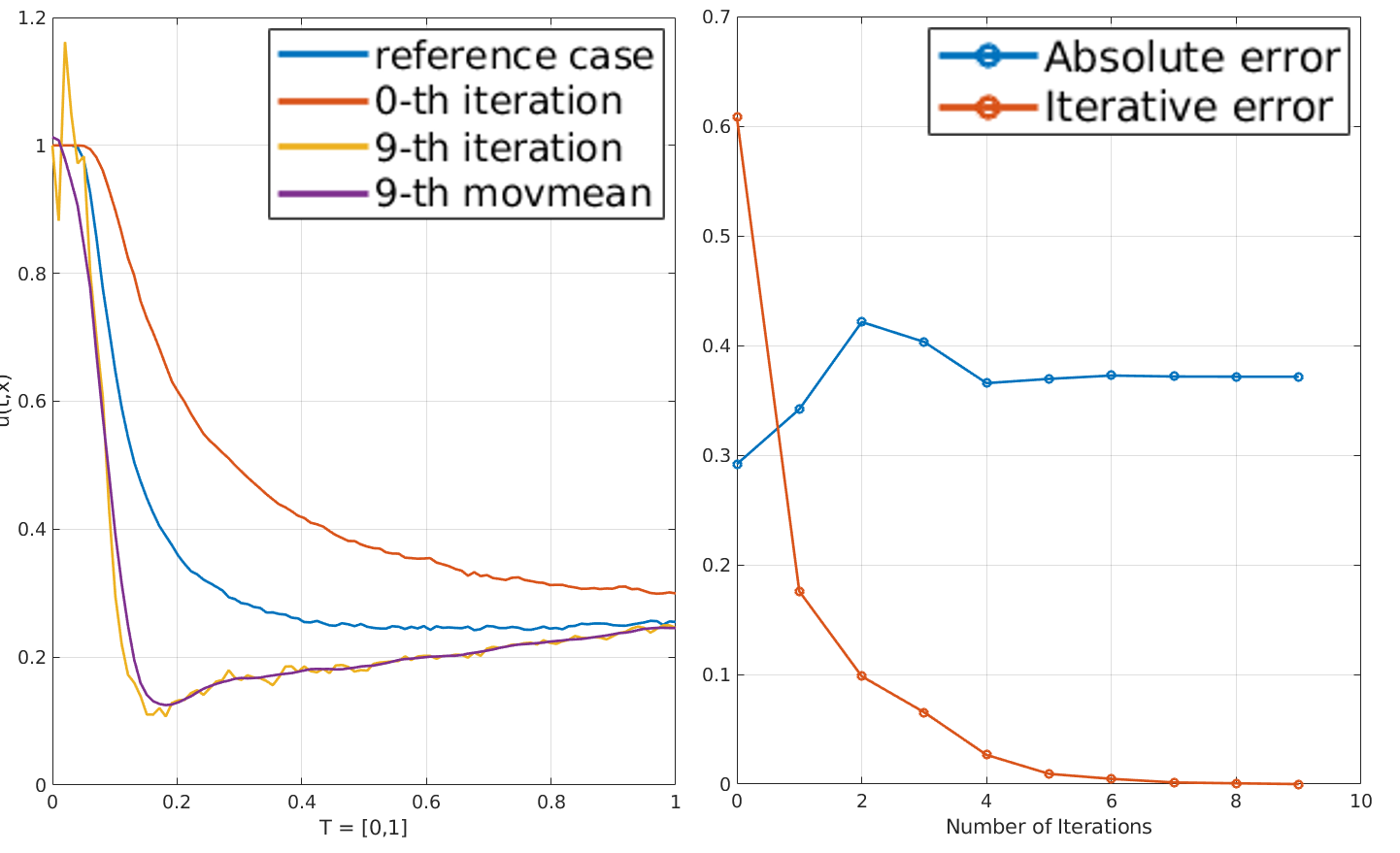}
\label{fig:sinBmD50}
\end{subfigure}
\caption{Sine times skew symmetric matrix \eqref{eq:sinSkew}. Left dimension $d = 20$, $N_{s} = 10^{4}$. Right dimension $d = 50$, $N_{s} = 10^{4}$. }
\label{fig:sinBmD20+sinBmD50}
\end{figure}

\paragraph*{Positive results.} For the simpler test case, the sine case \eqref{eq:sin}, see Figure \ref{fig:sinFull}, convergence is obtained in five iterations.  This is due to the simplicity of the example, as $\sin(x)$ is almost linear near the origin. The situation is different when dealing with some more concrete examples like the polynomial case. In Figure \ref{fig:quadraticfull}, where we use formula \eqref{eq:polynomialbounded} with $p=2$, we see that the number of iterations to convergence is much bigger (26 in our example). At the same time the difference between the last iteration and the reference case is quite small, comparable to the sine case. However, we notice that the oscillation of the solution computed via our iteration scheme, related to the variance of the estimator, is a bit bigger than that of the reference case. This discrepancy is not completely clear yet, even if we expect it to be due to the low number of samples used to compute averages. In Figure \ref{fig:quadraticfull} we also add a moving-average smoothing of the solution, to make more perceivable this last intuition.

The same behavior is obtained in the variations of the previous examples. In Figure \ref{fig:sinBmFull+cubicFullBad} we see that the same fast convergence as in the sine case, is obtained also in the sine times skew-symmetric case \eqref{eq:sinSkew}. The Polynomial cubic case \eqref{eq:polynomialbounded} with $p=3$ has the same level of complexity as the case with $p=2$, even if it requires a higher number of iterations to obtain convergence, and presents the same type of oscillations.

We also perform the same tests in much higher dimension. In Figure \ref{fig:sinD50+QuadraticD50} we show the results of the same examples, performed in dimension $d = 50$ with $N_{s} = 10^{4}$ samples. We see that the number of iterations required to convergence are comparable with result in $d = 10$: this confirms the estimate of Corollary \ref{cor-estimate} which is in the infinite dimensional framework and hence is independent of any dimension. The small variations in the number of iterations, as well as the slight increase of the oscillations in the quadratic case, can be explained by the reduction in the number of samples used to compute empirical averages. It is also important to remark that, in the current example, the estimate of Corollary \ref{cor-estimate} is still too rough: by computing the right-hand side of \eqref{eq:cor-estimate} one finds that the number $n$ of iterations needed to have $|v^{n}(t,x)| < tol$ is far bigger than what we find in the numerical test (in fact it should be bigger than one hundred).

\begin{figure}[]
\centering
\includegraphics[width=\textwidth]{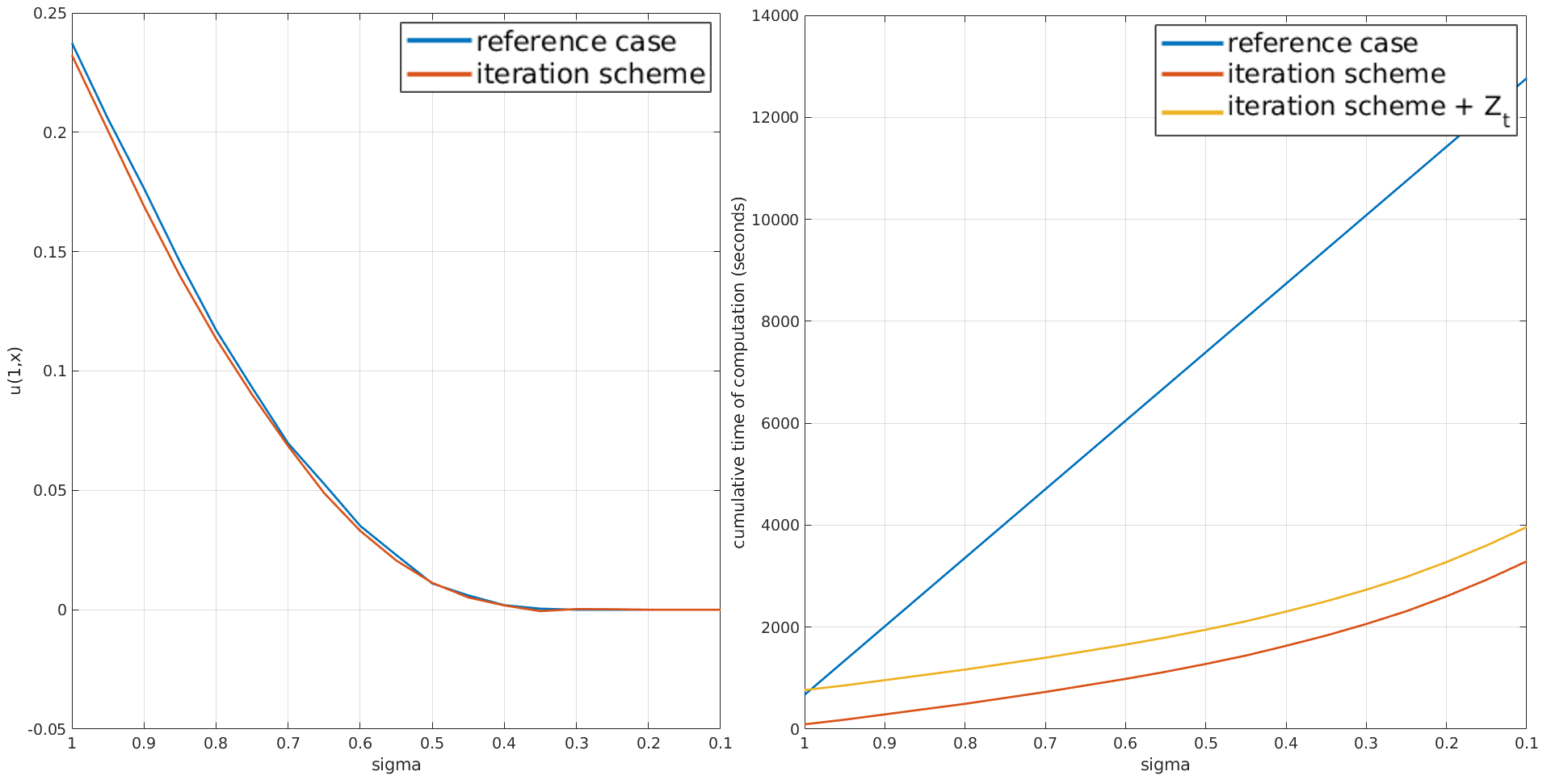}
\caption{Sine times skew symmetric matrix \eqref{eq:sinSkew}, $d=10$. Left: Y-axis value of $u(1,x)$ for different values of $\sigma$. X-axis different values of $\sigma$ in the reverse order. Right: Y-axis  computational time measured in seconds to compute the solution $u(1,x)$ for various values of $\sigma$. The measurement of time is cumulative: we give the cost of computing $u(1,x)$ for several values of $\sigma$, starting from $\sigma = 1$ in decreasing order. X-axis different values of $\sigma$ in the reverse order. The red line refers only to the time to compute iterations. The yellow line includes also the time to compute samples of $Z_{t}$ one time at the beginning of the simulation.}
\label{fig:longsigma}
\end{figure}

In Figure \ref{fig:longsigma} we followed a different approach: we fix the test case as the sine times skew-symmetric matrix \eqref{eq:sinSkew}, and analyze what is the limit of $u(1,x)$ as $\sigma$ goes to zero. Also in this case the solution computed through the iteration scheme is quite close to the reference case. At the same time, on the right side of Figure \ref{fig:longsigma}, we can appreciate the great advantage in time-saving of the iteration scheme. We remark that the plot on the right side is cumulative, meaning that it takes into account the time spent to compute the solution multiple times. In particular, we note that the reference case is a straight line, since the computational time does not depend on the different values of $\sigma$. On the other hand, for the iteration scheme there is a change in the number of iterations for different values of $\sigma$ that justifies the nonlinear shape. Moreover we see that, even including the time of computing samples of the process $Z_{t}$ that can be done only once (since $\sigma$ does not appear in \eqref{eq:Znum}), we still have a great advantage in time.

As remarked in the introduction, this kind of advantage is a main feature of the new method proposed here and applies also to the variation of other parameters than $\sigma$. In particular, it applies to the change of initial conditions $x$, one of the most fundamental problems in weather and climate prediction, related to the ensemble forecasting method, see \cite[Chapter 6]{Kalnay}. Again, Monte Carlo pays linearly with the number of variations of $x$, while our method pays the bulk (i.e. $Z_t$ in \eqref{eq:Znum}) only once and then (here for the initial conditions) roughly linearly in the number of different $x$'s, but with a linear slope much smaller than the one of Monte Carlo, similarly to the initial slope of Figure \ref{fig:longsigma} right side. We illustrate the interest in varying $x$ by Figure \ref{fig:sigmaQuadraticBad+changeX} right side, where it is illustrated the relative importance of different variations.

\begin{figure}[t]
\begin{subfigure}{0.49\textwidth}
\includegraphics[width=\textwidth]{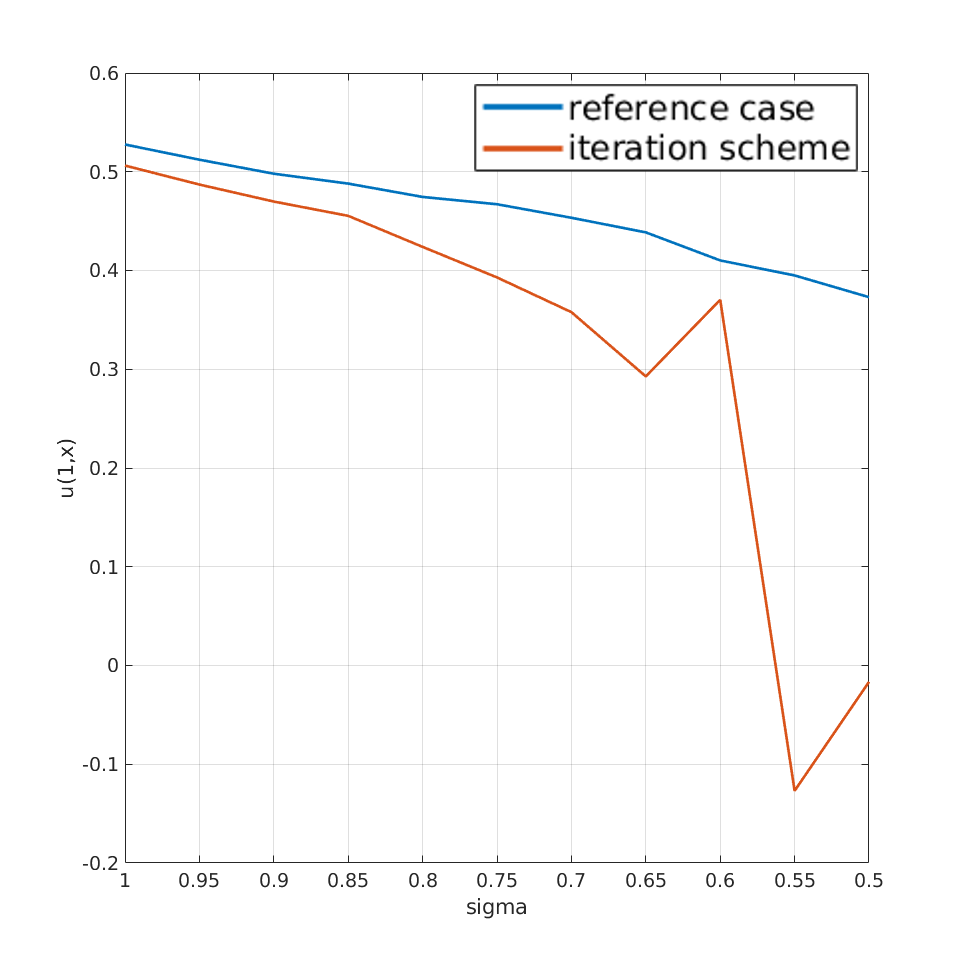}
\label{fig:sigmaQuadraticBad}
\end{subfigure}
\begin{subfigure}{0.49\textwidth}
\includegraphics[width=\textwidth]{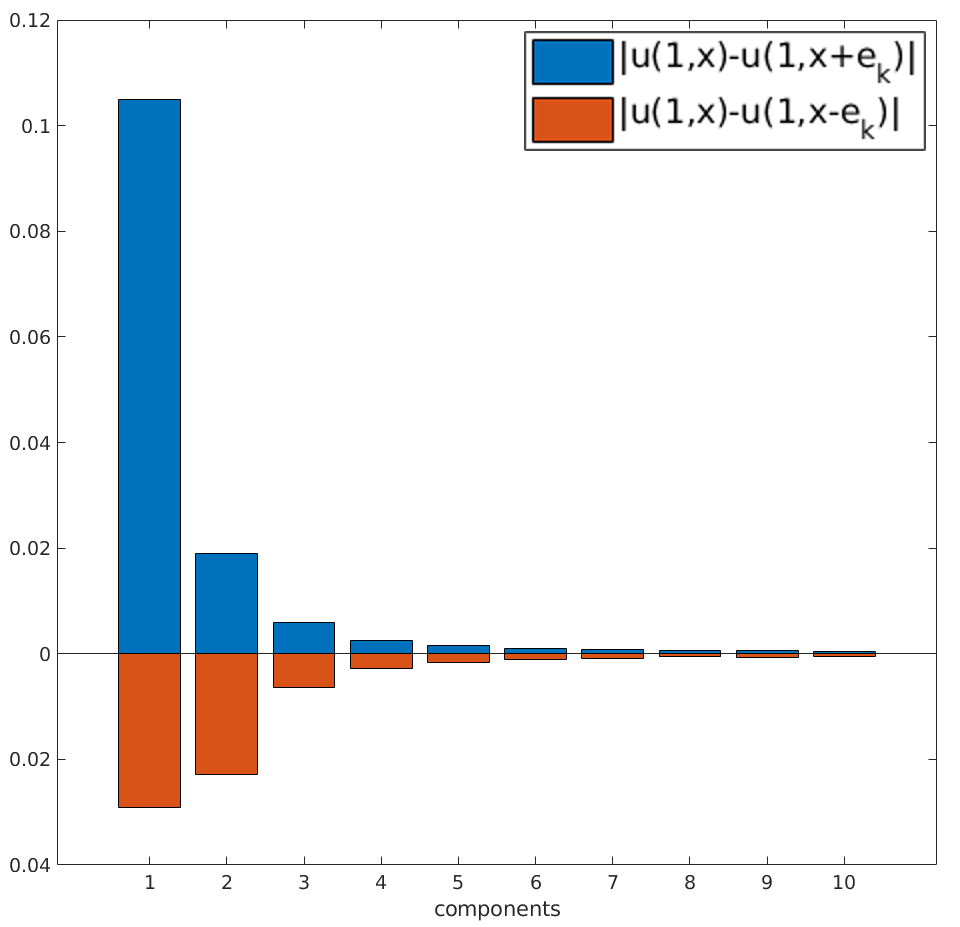}
\label{fig:changeX}
\end{subfigure}
\caption{Polynomial quadratic bounded case \eqref{eq:polynomialbounded}, $p=2$, $d = 10$. Left: Y-axis Value of $u(1,x)$ computed by the iteration scheme, with different values of $\sigma$. X-axis Different values of $\sigma$ in the reverse order. Right: Sine times skew symmetric matrix \eqref{eq:sinSkew}, $d = 10$. Difference of $u(1,x)$ with respect to $u(1, x \pm e_{k})$ for $k = 1,\dots, 10$. Blue positive values are obtained by comparing with $u(1,x+e_{k})$, orange negative by comparing with $u(1,x-e_{k})$. X-axis different values of $k = 1,\dots, 10$.}
\label{fig:sigmaQuadraticBad+changeX}
\end{figure}

\paragraph*{Difficulties with small $\sigma$ and high dimension.}
However, not every situation is well behaved as those presented above: in Figure \ref{fig:sigmaQuadraticBad+changeX} left side, we present the plot for different values of $\sigma$ in the polynomial quadratic case. Here the approximation tends to degenerate for smaller values of $\sigma$ (already around $0.5$). This is due to the higher level of nonlinearity of the polynomial case with respect to \eqref{eq:sinSkew}. It is also important to mention that the number of iterations to convergence is really important for what concerns the computational time. In the polynomial quadratic (and also cubic) case, since the number of iterations to convergence is much higher than in the simpler case, the advantage in the computational time is less relevant.  Still for what concerns negative results we also show in Figure \ref{fig:sinBmD20+sinBmD50} that, when the dimension grows (left $d = 20$, right $d=50$), the sine times skew-symmetric case \eqref{eq:sinSkew} tends to degenerate. Iterations are still converging but the limit is far form the reference solution. This is definitively the most difficult of our examples since it is the only one which mixes strongly all the components and produces a strong energy flux between them. However we also remark that at present time this case is not covered by our theory, but is still relevant since it has the rotational behavior which appeals to fluid dynamics.

\end{document}